\documentclass[final,hidelinks,onefignum,onetabnum]{siamart251216}

\usepackage{amsfonts}
\usepackage{graphicx}
\usepackage{hyperref}
\usepackage{algorithmic} 
\usepackage{paralist}
\usepackage{etoolbox}
\usepackage[bold]{hhtensor} 
\usepackage{tikz}
\usepackage{pgfplots}
\pgfplotsset{compat=1.18}

\ifpdf
  \DeclareGraphicsExtensions{.eps,.pdf,.png,.jpg}
\else
  \DeclareGraphicsExtensions{.eps}
\fi

\newsiamremark{remark}{Remark}
\newsiamremark{hypothesis}{Hypothesis}
\crefname{hypothesis}{Hypothesis}{Hypotheses}
\newsiamthm{claim}{Claim}
\newsiamremark{fact}{Fact}
\crefname{fact}{Fact}{Facts}

\headers{Gauge-invariant HHO method for magnetic Schr\"{o}dinger}{J. Aghili}

\title{Asymptotic gauge-invariant Hybrid High-Order method for magnetic Schr\"{o}dinger equations}

\author{Joubine Aghili\thanks{IRMA, Université de Strasbourg, CNRS UMR7501, 7 rue Ren\'{e} Descartes, 67084 Strasbourg, France
  (\email{aghili@unistra.fr}).}}

\usepackage{amsopn}

\usepackage{custom}

\ifpdf
\hypersetup{
  pdftitle={Asymptotic gauge-invariant Hybrid High-Order method for magnetic Schr\"{o}dinger equations},
  pdfauthor={Joubine Aghili}
}
\fi

\begin{document}

\maketitle

\begin{abstract}
We introduce a Hybrid High-Order (HHO) method for the Schr\"{o}dinger equation in the presence of a magnetic vector potential. In quantum mechanics, physical observables are invariant under continuous gauge transformations, which must be kept at the discrete level to avoid unphysical artifacts. To address this, we construct a discrete covariant gradient operator on arbitrary polyhedral meshes. We prove that the resulting discrete bilinear form guarantees gauge covariance asymptotically at the discrete level. The resulting scheme achieves optimal convergence rates and preserves a discrete Gårding inequality, guaranteeing a stable ground state. The theoretical properties of the scheme are corroborated by numerical experiments, including the computation of the Fock-Darwin fundamental energy and replicating the Aharonov-Bohm effect.
\end{abstract}



\section{Introduction}

The numerical simulation of quantum mechanical systems under the influence of electromagnetic fields plays a pivotal role in modern physics, with applications ranging from condensed matter to quantum computing.
The evolution and stationary states of a non-relativistic particle in a magnetic field are governed by the Schr\"{o}dinger equation, where the magnetic effect is introduced via a vector potential $\vA$.
A fundamental property of this formulation is its \textit{gauge invariance}: transforming the vector potential as $\vA \to \vA + \nabla \chi$ alongside a local phase shift of the wavefunction $\psi \to \psi e^{i\chi}$ leaves the physical observables unchanged.

From a numerical perspective, preserving this continuous symmetry at the discrete level is a non-trivial task. Standard discretization schemes often break gauge invariance, introducing artificial dependencies on the chosen gauge and potentially leading to spurious eigenvalues or unphysical long-time dynamics. The severe consequences of such discretizations were explicitly highlighted by Governale and Ungarelli \cite{Governale1998}, who pioneered the use of Wilson's lattice gauge theory on uniform finite difference grids to restore gauge invariance. This lattice gauge formalism has since been generalized to broader grid-based and pseudospectral methods by Halvorsen and Kvaal \cite{Halvorsen2009}, and recently adapted for the non-perturbative solution of diatomic molecules in strong magnetic fields by Yenugu et al. \cite{Yenugu2025}. For time-dependent problems, other grid-based approaches have also been explored, such as the Finite Difference Time Domain (FDTD) method for complex wavefunctions \cite{Sudiarta2007}, splitting methods utilizing rotated potentials to accommodate magnetic rotations \cite{Gradinaru2020}, and open-source split-step Fourier solvers \cite{QMsolve}.

In the realm of finite element methods (FEM), Alouges and Bonnaillie-Noël \cite{Alouges2006} proposed bypassing standard gauge issues by decoupling the modulus and phase of the wavefunction to study eigenstates in domains with corners. From a preconditioning perspective, Ovall and Zhu \cite{Ovall2025} recently demonstrated that computing a canonical gauge via a Helmholtz-Hodge decomposition can significantly reduce eigenvector oscillations, thereby improving standard FEM efficiency. A major breakthrough in structure-preserving FEM was achieved by Christiansen and Halvorsen \cite{Christiansen2011, Christiansen2015}, who successfully merged lattice gauge concepts with finite element exterior calculus to yield gauge-invariant discretizations on simplicial grids, achieving up to second-order convergence. Furthermore, in the context of Discontinuous Galerkin (DG) methods, Yi, Huang, and Liu \cite{Yi2020} developed conservative schemes that preserve mass and energy for nonlinear magnetic Schr\"{o}dinger equations while achieving optimal $L^2$ error estimates. 
Despite these notable advancements across various discretization paradigms, extending exact discrete gauge invariance to arbitrarily high-order methods on complex unstructured polyhedral meshes, using a purely local formulation remains, to the best of the author's knowledge, an open problem.

In recent years, the Hybrid High-Order (HHO) method \cite{DiPietro2014} has emerged as a highly flexible framework for the discretization of partial differential equations. By employing fully discontinuous polynomial spaces on mesh elements and faces, HHO methods seamlessly handle arbitrary polyhedral grids and provide dimension-independent constructions.

The primary contribution of this paper is the design and analysis of a gauge-invariant HHO method for the Schr\"{o}dinger equation. We achieve this by introducing a \textit{discrete covariant gradient} that intrinsically respects the coupling between the spatial derivative and the magnetic potential. Our approach offers the following main features:
\begin{inparaenum}[(i)]
    \item \textbf{Local construction:} The discrete covariant gradient is computed element-wise, on a general $d$-dimensional element, with a controlled accuracy.
    \item \textbf{Asymptotically discrete gauge covariance:} We prove that under a discrete equivalent of the gauge transformation, the local reconstructions and the global bilinear form behave covariantly, ensuring that the physical observables are invariant asymptotically.
    \item \textbf{Discrete G\aa rding inequality:} Due to the magnetic potential, the standard coercivity of the operator is lost. We bypass this issue by proving a gauge-invariant discrete G\aa rding inequality, which guarantees that the discrete system possesses a stable ground state and a spectrum bounded from below.
\end{inparaenum}
The rest of the paper is organized as follows. In \cref{sec:continuous}, we recall the continuous problem, establish the continuous G\aa rding inequality, and review the principles of gauge invariance. \Cref{sec:hho} introduces the HHO discretization, details the construction of the discrete covariant gradient, and presents the main theoretical proofs regarding gauge covariance and stability. Finally, \cref{sec:numerics} is devoted to numerical experiments on 2D and 3D unstructured meshes, validating the theoretical optimal convergence rates and the exactness of the discrete gauge invariance on the computation of the Fock-Darwin spectrum and the long-time behavior of the system.

\section{The Covariant Derivative and the Schr\"{o}dinger Equations}
\label{sec:continuous}

Let \(\vA\) denote a vector potential. The central focus of this paper is the covariant derivative operator \((i\GRAD - \vA)\). This operator naturally emerges in quantum mechanics when considering charged particles in a magnetic field \(\vB\), where \(B = \ROT \vA\), as dictated by the Maxwell-Thomson law. We shall explore two pertinent problems involving this operator: the stationary and unsteady Schr\"{o}dinger equations. For the sake of simplicity, we set \(\hbar = m = q = 1\) (the physical constants can be reintroduced without difficulty). 

Let \(\Omega \subset \Real^d\) (where \(d = 2,3\)) be a Lipschitz bounded domain, and consider a quantum particle \(P\) characterized by its wave function \(\psi: [0,T] \times \DOM \to \Complex\). The probability of finding the particle within an infinitesimal box \(dxdy\) at time \(t\) is given by the quantity \(|\psi(t,x,y)|^2 dxdy\), such that $\int_{\DOM}|\psi(t,x,y)|^2dxdy=1$ for any $t\in[0,T)$.

\paragraph{The stationary Schr\"{o}dinger Equation} Let \(\vA \in W^{1,\infty}(\Omega; \Real^d)\) be a vector potential and \(V \in L^\infty(\Omega; \Real)\) a scalar potential defined over \(\DOM\). The stationary Schr\"{o}dinger equation is expressed as an eigenvalue problem: find eigenpairs $\psi$ and $\lambda$, s.t. 
\begin{equation}
\begin{cases}
\bigl(-i\GRAD - \vA\bigr)^2 \psi + V\psi = \lambda \psi & \text{in } \Omega, \\
\psi = 0 & \text{on } \partial\Omega.
\end{cases}
\label{eq:Schrodinger.stationnary}
\end{equation}
Here, \(\bigl(-i\GRAD - \vA\bigr)^2\) is interpreted in the sense of $\cG^*\cG$, where $\cG$ and $\cG^*$ are defined as 
\[
  \cG u:=(-i\GRAD u -\vA u \bigr), \qquad \cG^*\vtau =-i\DIV \vtau - \vA \cdot \vtau,
\]
for any $u\in H^1_0(\DOM; \Complex)$ and $\vtau$ in $H^1_0(\DOM;\Complex)^d$.
The eigenvalues are guaranteed to be real, as $\cG^*\cG$ is self-adjoint; however, they may be nonpositive, with the lowest eigenvalue $\lambda_0$ being particularly significant in applications. 
The variational formulation entails identifying \(\lambda \in \Real\) and \(\psi \in H_0^1(\Omega; \Complex)\) such that 
\begin{equation}
a(\psi, \phi; \vA) = \lambda (\psi, \phi)_{L^2(\Omega)} \quad \forall \phi \in H_0^1(\Omega; \Complex), \qquad \|\psi\|_{L^2(\Omega)} = 1,
\end{equation}
where the sesquilinear form $a(\cdot,\cdot;\vA)$ is defined as
\begin{equation}
  a(\psi,\phi; \vA) \eqbydef \int_{\Omega} \cG u \cdot \overline{\cG v} \, dx + \int_{\Omega} V u \overline{v} \, dx.
  \label{eq:sesquilinear_form}
\end{equation}
We will simply write $a(\psi,\phi)$ to refer to $a(\psi, \phi; \vA)$ when the context is unambiguous.

\paragraph{The unsteady Schr\"{o}dinger Equation} Let \(\psi_0\) represent an initial quantum state and \(f \in L^2(\DOM; \Complex)\) denote a forcing term. The unsteady (or time-dependent) Schr\"{o}dinger equation is formulated as follows: find $\psi$ s.t.
\begin{equation}
\begin{cases}
i\partial_t \psi + \bigl(-i\GRAD - \vA\bigr)^2 \psi + V\psi = f & \text{in } \Omega, \\
\psi(0) = \psi_0 & \text{on } \partial\Omega.
\end{cases}
\label{eq:Schrodinger.time}
\end{equation}

\subsection{G\aa rding Inequality and Spectrum}

Before discussing the gauge invariance, we establish a fundamental property of the sesquilinear form \eqref{eq:sesquilinear_form}.
Although the form $a(\cdot, \cdot)$ is not strictly coercive due to the presence of the vector potential $\mathbf{A}$, it satisfies a G\aa rding inequality \cite{Garding1953}.
This property ensures that the quantum system has a stable ground state energy.

\begin{lemma}[G\aa rding Inequality]
  \label{lem:garding}
There exist constants $\alpha > 0$ and $\gamma \ge 0$, depending on $\|\mathbf{A}\|_{L^\infty(\Omega)}$ and $\|V\|_{L^\infty(\Omega)}$, such that for all $u \in H_0^1(\Omega;\mathbb{C})$,
\[
  a(u, u) \ge \alpha \|\nabla u\|_{L^2(\Omega)}^2 - \gamma \|u\|_{L^2(\Omega)}^2.
\]
\end{lemma}
\begin{proof}
By definition of the sesquilinear form, and taking the real part (which is equal to the form itself since it is Hermitian), we have:
$$a(u,u) = \int_{\Omega} |(-i\nabla - \mathbf{A})u|^2 dx + \int_{\Omega} V |u|^2 dx.$$
Using the elementary inequality $|a - b|^2 \ge \tfrac{1}{2}|a|^2 - |b|^2$ for vectors in $\mathbb{C}^d$, we can lower bound the kinetic energy term:
$$|(-i\nabla - \mathbf{A})u|^2 \ge \frac{1}{2}|\nabla u|^2 - |\mathbf{A}u|^2.$$
Integrating this inequality over $\Omega$ and bounding the potential terms by their $L^\infty$ norms yields:
$$a(u,u) \ge \tfrac{1}{2} \|\nabla u\|_{L^2(\Omega)}^2 - \|\mathbf{A}\|_{L^\infty(\Omega)}^2 \|u\|_{L^2(\Omega)}^2 - \|V\|_{L^\infty(\Omega)} \|u\|_{L^2(\Omega)}^2.$$
The result follows by choosing $\alpha = \frac{1}{2}$ and $\gamma = \|\mathbf{A}\|_{L^\infty(\Omega)}^2 + \|V\|_{L^\infty(\Omega)}$.
\end{proof}
\begin{remark}[Boundedness of the spectrum]
This inequality directly implies that the spectrum of the continuous problem is bounded from below. Indeed, if $(\lambda, \psi)$ is a solution to the eigenvalue problem with $\|\psi\|_{L^2(\Omega)} = 1$, then $\lambda = a(\psi, \psi)$. Applying the G\aa rding inequality, we obtain:
$$\lambda \ge \frac{1}{2} \|\nabla \psi\|_{L^2(\Omega)}^2 - \gamma \|\psi\|_{L^2(\Omega)}^2 \ge -\gamma.$$
The minimal energy of the system is therefore bounded by below by $-\gamma$. The discrete covariant gradient operator presented in \cref{sec:discrete.gradient} is specifically designed to preserve this property.
\end{remark}

\subsection{Gauge Invariance}

\begin{definition}
  \label{def:continuous_gauge_transformation}
  Let \(\chi \in W^{2,\infty}(\Omega;\Real)\) be an arbitrary gauge function. We define the \textit{gauge transformation} for any smooth function $\psi$ and vector field $\vA$ as
  \begin{equation}
    \psi_{\chi} := \psi \, e^{i\chi}, \qquad \vA[\chi] := \vA + \GRAD\chi.
  \end{equation}
  
\end{definition}

\begin{lemma}[Continuous Gauge Covariance]
  \label{lem:continuous_gauge_covariance}
  We have the following compatibilities with respect to gauge transformation,
\[
  \cGchi\psi_{\chi} = e^{i\chi} \cG\psi, \quad    e^{i\chi} \cG^* (e^{-i\chi} \vtau) = \cGchi^* \vtau,
\]
holds for any smooth functions $\psi$ and $\vtau$, almost everywhere in \(\Omega\). Consequently, the bilinear form $a(\cdot,\cdot)$, the probability density $|\psi|^2$, and the spectrum remain invariant under this transformation.
\end{lemma}

\begin{proof}
A direct computation gives $-i\GRAD(\psi e^{i\chi}) = e^{i\chi} (-i\GRAD\psi + \psi \GRAD\chi)$. Thus,
\[
(-i\GRAD - \vA - \GRAD\chi)(\psi e^{i\chi}) = e^{i\chi} (-i\GRAD - \vA)\psi \Longleftrightarrow \cGchi \psi_{\chi} = e^{i\chi}\cG \psi.
\]
The adjoint covariant operator satisfies the exact phase-shifting identity:
\begin{align*}
  e^{i\chi} \cG^* (e^{-i\chi} \vtau) &= e^{i\chi} \Big( -i\GRAD\cdot (e^{-i\chi} \vtau) - \vA \cdot (e^{-i\chi} \vtau) \Big) \\
  &= e^{i\chi} \Big( -\GRAD\chi \cdot (e^{-i\chi} \vtau) - i e^{-i\chi} \GRAD\cdot \vtau - \vA \cdot (e^{-i\chi} \vtau) \Big) \\
  &= -i\GRAD\cdot \vtau - (\vA + \GRAD\chi) \cdot \vtau = \cGchi^* \vtau.
\end{align*}
From the first equality, we can readily infer $|\psi_{\chi}|^2=|\psi|^2$ and 
\begin{align*}
  a(\psi_{\chi}, \phi_{\chi}; \vA[\chi]) &= \int_{\DOM}\cGchi \psi_{\chi} \cdot \overline{\cGchi \phi_{\chi}}dx +  \int_\DOM V \psi_{\chi} \overline{\phi_{\chi}} dx \\
                        &= \int_{\DOM}e^{i\chi}\cG[\vA]\psi \cdot e^{-i\chi}\overline{\cG[\vA]\phi}dx +  \int_\DOM V e^{i\chi}\psi e^{-i\chi}\overline{\phi} dx = a(\psi,\phi; \vA). 
\end{align*}
\end{proof}

\section{The Hybrid High-Order Discretization}
\label{sec:hho}

In the following discussion, we give a succinct overview of the Hybrid High-Order (HHO) method and expand its discrete gradient to include the discrete counterpart of the covariant gradient $\cG$. To ensure consistency with the notation commonly used in the HHO literature, we will denote the solution to the primary problem as $u$, reserving the symbol $\psi$ for contexts where the focus is predominantly on the physical aspects.

\subsection{Admissible meshes and local polynomial spaces}
\label{sec:mesh.polynomial.space}

Denote by ${\cH}\subset \Real_*^+ $ a countable set of meshsizes having $0$ as its unique accumulation point.
We consider $h$-refined spatial mesh sequences $(\Th)_{h \in \cH}$ where, for all $ h \in \cH $, $\Th=\{T\}$ is a polytopal tessellation of $\DOM$ such that $h=\max_{T\in\Th} h_T$ with $h_T$ standing for the diameter of the tile $T$.
We assume that mesh regularity holds in the sense of~\cite[Definition~1.9]{DiPietro2020}; cf.~\cite[Chapter~1]{DiPietro2020} for a collection of useful geometric and functional inequalities that hold under this assumption.

We define a mesh face $F$ as a hyperplanar closed connected subset of $\bar{\Omega}$ with positive $ (d{-}1) $-dimensional Hausdorff measure and such that%
\begin{inparaenum}[(i)]
\item either there exist $T_1,T_2\in\Th $ such that $F\subset\partial
  T_1\cap\partial T_2$ and $F$ is called an interface or 
\item there exists $T\in\Th$ such that $F\subset\partial T\cap\partial\Omega$ and $F$ is called a boundary face.
\end{inparaenum}
Interfaces are collected in the set $\Fhi$, boundary faces in $\Fhb$, and we let $\Fh\eqbydef\Fhi\cup\Fhb$.
For all $T\in\Th$, $\Fh[T]\eqbydef\{F\in\Fh\st F\subset\partial T\}$ denotes the set of faces contained in $\partial T$ and, for all $F\in\Fh[T]$, $\vn$ is the unit normal to $F$ pointing out of $T$.

We assume throughout the rest of this work that the mesh sequence $(\Th)_\cH$ is \emph{admissible} in the sense of~\cite[Chapter~1]{DiPietro2020}.
\begin{definition}[Admissible mesh sequence]
  \label{def:adm.meshes}
  For all $h\in\cH$, $\Th$ admits a matching simplicial submesh $\fTh$ and the following properties hold for all $h\in\cH$ with mesh regularity parameter $\varrho>0$ independent of $h$:
\begin{inparaenum}[(i)]
\item  for all simplex $S\in\fTh$ of  diameter $h_S$ and inradius $r_S$, $\varrho h_S\le r_S$;
\item for all $T\in\Th$, and all   $S\in\fTh[T]\eqbydef\{S\in\fTh\ST S\subset T\}$,   $\varrho h_T \le h_S$.
\end{inparaenum}
\end{definition}

For an admissible mesh sequence, it is known from~\cite[Lemma~1.12]{DiPietro2020} that the number of faces of one tile can be bounded uniformly in $h$, i.e., it holds that
\begin{equation}
  \label{eq:Np}
  \forall h\in\cH,\qquad
  \max_{T\in\Th}\card{\Fh[T]} \le\Np,
\end{equation}
for an integer $(d+1)\le\Np <+\infty$ depending on $\varrho$ but independent of $h$.
Furthermore, for all $h \in \cH$, all $T \in \Th$ and all $F \in \Fh[T]$, $h_F$ is uniformly  comparable to $h_T$ in the following sense:
$$\rho^2 h_T \le h_F \le h_T.$$

Let $X$ be a subset of $\Real^{N}$, $N\ge 1$, $H_X$ the affine space spanned by $X$, $d_X$ its dimension, and assume that $X$ has a non-empty interior in $H_X$ (in what follows, the cases $X=T$ and $X=F$ are relevant).
For a given integer $k\ge 0$, we denote by $\Poly{k}(X;\Complex)$ the space spanned by $d_X$-variate complex-valued polynomials on $H_X$ of total degree $\le l$ and by $\lproj[X]{k}: L^1(X) \to \Poly{k}(X;\Complex)$ the $L^2$-orthogonal projector on $\Poly{k}(X;\Complex)$.
We define the global interpolation operator $I_h^k: H^1(\Omega) \to U_h^k$ such that its local restriction is $I_T^k u = (\pi_T^k u, \pi_{\partial T}^k u)$, where $\pi_{\partial T}^k$ is the piecewise $L^2$-orthogonal projector on the faces of $T$.
The usual $L^2(X)$-product and norm are denoted by $(u,v)_X:=\int_{X}u\overline{v}dx$ and $\norm[X]{{\cdot}}$, respectively, and we omit the index when $X=\Omega$.
We also recall the following local trace and inverse inequalities (cf.~\cite[Section 1.2.5]{DiPietro2020}):
For all $T\in\Th$ and all $v\in\Poly{k}(T)$,
\begin{equation}\label{eq:trace.inv}
  \norm[F]{v} \lesssim h_F^{- \frac12} \norm[T]{v}\mbox{ for all $F\in\Fh[T]$ and }
  \norm[T]{\GRAD v}\lesssim h_T^{-1}\norm[T]{v}.
\end{equation}

\subsection{The degrees of freedom spaces}
\label{sec:dof.spaces}

We follow the standard Hybrid High-Order formulation for the Poisson problem. For a polynomial degree \(k \geq 0\), the local space of degrees of freedom (DOF) on each cell \(T\) is
\[
\UT := \Poly{k}(T;\Complex) \times \prod_{F\subset \partial T} \Poly{k}(F;\Complex),
\]
The global DOF space is denoted \(\Uh\), with homogeneous Dirichlet boundary conditions stongly enforced:
\[
  \Uh := \left\{ \uuh=(\uuTF)_{T\in \Th}, \quad \uuTF\in \UT, \quad u_F = 0, \forall F\in \Fhb \right\}.
\]
We define the following norm $\norm[1,h]{\cdot}$ over $\Uh$ as as
\begin{equation}
    \label{eq:discrete.norms}
    \norm[1,h]{\uuh}^2:= \sum_{T\in \Th}\norm[1,T]{\uuTF}^2, \qquad  \norm[1,T]{\uuTF}^2:= \norm{\GRAD u_T}^2 + \sum_{F\in\FT} h_F^{-1}\norm{u_F-u_T}^2.
\end{equation}

\subsection{Potential Reconstruction}
\label{sec:potential.reconstruction}

For \(\uuTF \in \UT\), the reconstructed potential \(\pT \uuTF \in \Poly{k+1}(T;\Complex)\) is defined such that for all \(w \in \Poly{k+1}(T;\Complex)\),
\begin{equation}
  \label{eq:pT.def}
(\GRAD \pT \uuTF, \GRAD w)_T = -(\uT, \Delta w)_T + \sum_{F\subset\partial T} (\uF, \GRAD w\cdot\vn)_F,
\end{equation}
or equivalently
\begin{equation}
  \label{eq:pT.equivalent}
  (\GRAD \pT \uuTF, \GRAD w)_T = (\GRAD \uT, \GRAD w)_T + \sum_{F\subset\partial T} (\uF-\uT, \GRAD w\cdot\vn)_F,
\end{equation}
closed with the normalization $(\pT\uuTF - u_T, 1)_T = 0$.
We extend this operator globally to $\Uh$ by applying it element-wise, namely we define $\ph:\Uh \to \Poly{k+1}(\DOM; \Complex)$ such that for any $T\in\Th$, 
\[
  (\ph \uuh)_{|T} := \pT \uuTF.
\]

\begin{proposition}
  \label{prop:pT.coercive}
  There exists $C > 0$ independent from $h$, but from the number of faces and the
  trace constant $C_{tr}$ such that
  \begin{equation}
    \norm{\GRAD \pT \uuTF}^2 + s_T(\uuTF, \uuTF) \geq C\norm{\GRAD \uT}^2.
  \end{equation}
\end{proposition}
\begin{proof}
  If $\uT$ is constant over $T$, the proof is valid as the LHS is always positive.
  We now assume $\uT$ is not constant over $T$. Using \eqref{eq:pT.equivalent}, we have
  \[
    \norm{\GRAD \uT}^2
    = (\GRAD \pT \uuTF,\, \GRAD \uT)_T
      - \sum_{F \subset \partial T} (\uF - \uT,\, \GRAD \uT \cdot \vn)_F
    = \term[1] +\term[2]
  \]
  where the first term $\term[1]$ can be readily bounded using Cauchy-Schwarz
  \[
    (\GRAD \pT \uuTF,\, \GRAD \uT)_T
    \leq \norm[T]{\GRAD \pT \uuTF}\norm[T]{\GRAD u},
  \]
  and the second term $\term[2]$ can be bounded as
  \begin{align*}
    \term[2]
    &\leq \sum_F \norm[F]{\uF - \uT}\norm[F]{\GRAD u \cdot \vn} \leq \left(\sum_F h_F^{-1}\norm[F]{\uF - \uT}^2\right)^{1/2}\!\left(\sum_F h_F\norm[F]{\GRAD \uT}^2\right)^{1/2} \\
    &\leq s_T(\uuTF, \uuTF)^{1/2} \left(\sum_F h_F^1 h_F^{-1} C_{tr}^2\norm[T]{\GRAD \uT}^2\right)^{1/2}  \leq s_T(\uuTF, \uuTF)^{1/2}\, C_{tr}\norm[T]{\GRAD \uT}\Np^{1/2} \\
    &\leq C_{tr}\Np^{1/2}\, s_T(\uuTF, \uuTF)^{1/2}\norm[T]{\GRAD \uT}.
  \end{align*}
  Hence, dividing by $\norm{\GRAD \uT} \ne 0$, we have
  $\norm[T]{\GRAD \uT} \lesssim s_T(\uuTF, \uuTF)^{1/2}$.
  Putting squares on the above inequality and using $(a+b)^2 \leq 2a^2 + 2b^2$
  and adding the stabilization terms, we finally obtain
  \[
    C\norm[1,T]{\uuTF}^2 \leq \norm[T]{\GRAD \pT \uuTF}^2 + s_T(\uuTF, \uuTF),
  \]
  with $C := \dfrac{1}{\max(2,\, 2C_{tr}\Np^{1/2}+1)}$, independent of $h$.
\end{proof}

\subsection{Discrete Covariant Gradient and Bilinear Form}
\label{sec:discrete.gradient}

Let $\vA \in L^{\infty}(\DOM)^d$ any smooth vector potential and $\vA_{\chi} \in L^{\infty}(\DOM)^d$ its transformation with $\chi \in W^{1,\infty}(\overline{\DOM})$ a real-valued gauge phase.
Denoting $\uvh = (\uuTF)_{T\in \Th}\in \Uh$ an arbitrary DOF, $\uvh[\chi]$ and $\vA[T,\chi]$ the DOFs, and discrete vector potential, after a \textit{discrete gauge transformation} transformation as
\[
  \uvTF[\chi]:= \left( \projT(e^{i\chi} v_T), \projF(e^{i\chi} v_F) \right), \quad \vA[\chi,T]:=\projT \vA[\chi].
\]
We define the \textit{discrete covariant gradient} $\GT[\vA]$, for any vector potential $\vA$, as 
  \begin{equation}
  \label{eq:def.GT}
  (\GT[\vA] \uuTF, \vtau)_T = (u_T, \cG^*\vtau)_T - i\sumF (v_F, \vtau\cdot \vn)_F, \qquad \forall \vtau\in\Poly{k}(T)^d,
\end{equation}
or equivalently, for any $\forall \vtau\in\Poly{k}(T)^d$,
\begin{equation}
  \label{eq:def.GT.adjoint}
  (\GT[\vA] \uuTF, \vtau)_T = (\cG u_T, \vtau)_T - i\sumF (v_T-v_F, \vtau\cdot \vn)_F.
\end{equation}  

The local discrete bilinear form $a_T$, for any cell element $T\in\Th$, is defined using the $L^2$ projection of $\vA$, as 
\begin{equation}
  a_T(\uuTF, \uvTF; \vAT) \eqbydef \bigl(\GT \uuTF, \GT \uvTF \bigr)_{L^2(T)} + s_T(\uuTF, \uvTF),
  \label{eq:aT.defn}
\end{equation}
where \(s_T(\cdot,\cdot)\) denotes a suitable HHO stabilization, in the sense of \cite[Assumption 2.4, p. 49]{DiPietro2020}, for instance the choice from \cite[Eq. (2.22)]{DiPietro2020}:
\[
  s_T(\uuTF, \uvTF) := \sumF h_F^{-1}((\delT[TF] - \delT)\uuTF, (\delT[TF] - \delT)\uvTF)_F. 
\]
where $\delT\uuTF:= \projT(\pT \uvTF - \vT)$ and $\delT[TF]\uuTF := \projF(\pT \uvTF - \vF)$ for all $F\in \FT$.
The global bilinear form over $\Th$ is defined as the sum of all the local contributions \eqref{eq:aT.defn}, i.e. 
\begin{align}
  \label{eq:def.ah}
  a_h(\uuh, \uvh; \vA[h]) &:= \sumT a_T(\uuTF, \uvTF; \vAT),
\end{align}
where $\vA[h]:=\projh (\vA)=(\projT \vA)_{T\in\Th}$ is the $L^2$ projection of $\vA$ over all $T\in\Th$.

\begin{theorem}[Discrete Gauge Covariance]
\label{thm:gauge_covariance}
Let $\chi \in C^\infty(\overline{\DOM})$ be a smooth real-valued gauge phase, and $\GT[\chi]:=\GT[{\vAT[\chi]}]$. Then, for any $\uvh \in \Uh$, we have
\[
  ( \GT[\chi] (\uvTF[\chi]), \vtau )_T = ( \GT[\vA] \uvTF, \projT(e^{-i\chi} \vtau) )_T + \Rgauge (\uvTF, \vtau) \qquad \forall \vtau \in \Poly{k}(T)^d,
\]
where the remainder $\Rgauge$ gathers all projection commutators and is bounded by:
\[
  |\Rgauge(\uvTF, \vtau)| \le C h_T^{k+1} \|\uvTF\|_{1,T} \|e^{-i\chi}\vtau\|_{H^{k+1}(T)}.
\]
\end{theorem}

\begin{proof}
Let $\vtau \in \Poly{k}(T)^d$. By definition of the discrete gradient $\GT[\chi]$ with the projected potential $\vAT[\chi]$, we have for the first term:
\begin{align}
  (\GT[\chi] \uvTF[\chi], \vtau)_T &= (\pi_T^k(e^{i\chi} v_T), \cGATchi^*\vtau)_T - i \sumF (\pi_F^k(e^{i\chi} v_F), \vtau \cdot \vn )_F \\
  &= (\pi_T^k(e^{i\chi} v_T), \cGATchi^*\vtau)_T - i \sumF ( e^{i\chi} v_F, \vtau \cdot \vn )_F   \label{proof:gauge_convariance:1}
\end{align}
where we have used the fact that $\vtau \cdot \vn \in \Poly{k}(F)$, so that the face projector $\pi_F^k$ can be exactly dropped by orthogonality.
For the volumetric term, we add and subtract the \textit{continuous} potential $\vA[\chi]$ to reconstruct the continuous adjoint $\cGchi^* \vtau = -i\DIV \vtau - \vA[\chi] \cdot \vtau$:
\[
  \cGATchi^* \vtau = -i\DIV \vtau - \vAT[\chi] \cdot \vtau = \cGchi^* \vtau + (\vA[\chi] - \vAT[\chi]) \cdot \vtau.
\]
Substituting the above equality in \eqref{proof:gauge_convariance:1} yields the decomposition:
\begin{multline}
  (\GT[\chi] \uvTF[\chi], \vtau)_T = (\pi_T^k(e^{i\chi} v_T), \cGchi^* \vtau)_T + (\pi_T^k(e^{i\chi} v_T),(\vA[\chi] - \vAT[\chi]) \cdot \vtau)_T \\ - i \sumF ( e^{i\chi} v_F, \vtau \cdot \vn )_F.
\end{multline}
Adding and removing $e^{i\chi}\vT$ in the first volumetric term in the RHS gives
\begin{align*}
  (\GT[\chi] \uvTF[\chi], \vtau)_T =\; & (e^{i\chi} v_T, \cGchi^*\vtau)_T - i \sumF ( e^{i\chi} v_F, \vtau \cdot \vn )_F \\
  & + (\pi_T^k(e^{i\chi} v_T), (\vA[\chi] - \vAT[\chi]) \cdot \vtau)_T \\
  & + (\pi_T^k(e^{i\chi} v_T) - e^{i\chi}v_T, \cGchi^* \vtau)_T  =:\; \term[1] + \term[2] + \term[3].
\end{align*}

We now show that $\term[1]$ simplifies.
From Lemma \ref{lem:continuous_gauge_covariance}, we have $\cGchi^* \vtau = e^{i\chi}\cG^* (e^{-i\chi} \vtau)$, the global phase cancels perfectly in the complex inner product:
\begin{align*}
  \term[1] &= (e^{i\chi} v_T, e^{i\chi}\cG^* (e^{-i\chi} \vtau))_T - i \sumF (e^{i\chi} v_F, e^{i\chi}(e^{-i\chi}\vtau) \cdot \vn )_F \\
           &= (v_T, \cG^* (e^{-i\chi} \vtau))_T - i \sumF (v_F, (e^{-i\chi}\vtau) \cdot \vn )_F.
\end{align*}

Let $\vE := e^{-i\chi}\vtau - \projT(e^{-i\chi}\vtau)$ be the projection error of the gauged test function. Substituting $e^{-i\chi}\vtau = \projT(e^{-i\chi}\vtau) + \vE$ into $\term[1]$, we separate the polynomial part from the remainder:
\begin{align*}
  \term[1] =\; & (v_T, \cG^* \projT(e^{-i\chi}\vtau))_T - i \sumF (v_F, \projT(e^{-i\chi}\vtau) \cdot \vn )_F \\
  & + (v_T, \cG^* \vE)_T - i \sumF (v_F, \vE \cdot \vn )_F.
\end{align*}

By the exact definition of the discrete gradient $\GT[\vA]$ given in \eqref{eq:def.GT}, the first line is exactly equal to $(\GT[\vA] \uvTF, \projT (e^{-i\chi}\vtau))_T$. Thus, we have:
\[
  \term[1] = (\GT[\vA] \uvTF, \projT (e^{-i\chi}\vtau))_T + \term[\vE],
\]
where $\term[\vE]$ encapsulates the action of the continuous operator on the projection error $\vE$:
\[
  \term[\vE] := (v_T, \cG^* \vE)_T - i \sumF (v_F, \vE \cdot \vn )_F.
\]

Denoting the total remainder $\Rgauge(\uvTF, \vtau) := \term[2] + \term[3] + \term[\vE]$, it remains to bound these terms. 
For $\term[2]$ and $\term[3]$, standard approximation properties of the $L^2$-projector yield the optimal bound $\bigO(h_T^{k+1}) \|v_T\|_T \|\vtau\|_T$.
The term $\term[\vE]$ is treated by applying integration by parts to the continuous adjoint $\cG^* \vE = -i\DIV \vE - \vA \cdot \vE$:
\begin{align*}
  \term[\vE] &= (v_T, -i\DIV \vE)_T - (v_T \vA, \vE)_T - i \sumF (v_F, \vE \cdot \vn )_F \\
           &= (-i\GRAD v_T, \vE)_T - (v_T \vA, \vE)_T + i \sumF \langle v_T - v_F, \vE \cdot \vn \rangle_F.
\end{align*}
Because $v_T \in \Poly{k}(T)$, its gradient $\GRAD v_T$ is a polynomial of degree $k-1$. Since $\vE$ is the error of the $L^2$-orthogonal projection onto $\Poly{k}(T)^d$, the term $(-i\GRAD v_T, \vE)_T$ vanishes identically by orthogonality. We are left with purely $L^2$ terms which can be estimated:
\begin{align*}
  |\term[\vE]| &\le \|\vA\|_{L^\infty(T)} \|v_T\|_T \|\vE\|_T + \sumF \|v_T - v_F\|_F \|\vE\|_F \\
             &\le C \|\vA\|_{L^\infty(T)} \|v_T\|_T \left( h_T^{k+1} |e^{-i\chi}\vtau|_{H^{k+1}(T)} \right) + \sumF |v_h|_{1,T} h_F^{\frac12} \left( h_T^{k+\frac{1}{2}} |e^{-i\chi}\vtau|_{H^{k+1}(T)} \right) \\
             &\le C h_T^{k+1} |v_h|_{1,T} \|e^{-i\chi}\vtau\|_{H^{k+1}(T)},
\end{align*}
which concludes the proof.
\end{proof}

\begin{remark}
By employing the adjoint operator $(\projT)^*$ and defining $\vAT[\chi] := \projT(\vA[\chi])$, the preceding identity can be recast as
\[
\GT[{\vAT[\chi]}] \, \uvTF[\chi] 
= e^{i\chi} \, (\projT)^* \GT[\vA] \, \uvTF + \bigO(h^{k+1}),
\]
which parallels the continuous counterpart $\cG[{\vA[\chi]}] \, \psi_{\chi} = e^{i\chi} \, \cG \psi$.
\end{remark}

\begin{remark} Setting  $\vtau_h := \GT[\vA]\uvh \in \Poly{k}(T)^d$ and evaluating the $(k+1)$-th derivatives of the gauged test function via the Leibniz rule yields:
\[
  |e^{-i\chi} \vtau_h|_{H^{k+1}(T)} \le C_\chi \sum_{j=1}^{k+1} \|\GRAD^{k+1-j} \vtau_h\|_{T}.
\]
Since $\vtau_h$ is a polynomial of degree $k$, its derivatives up to order $k$ are non-zero.
Using the inverse inequality yields $\|\GRAD^m \vtau_h\|_T \le C_{inv} h_T^{-m} \|\vtau_h\|_T$. The dominant term corresponds to $j=1$ (the $k$-th derivative of $\vtau_h$), which introduces a factor $h_T^{-k}$. Consequently, the general algebraic covariance bound degrades to first-order:
\[
  |\Rgauge(\uvh, \vtau_h)| \lesssim h_T^{k+1} \left( h_T^{-k} \|\vtau_h\|_T \right) \|\uvh\|_{1,T} = \bigO(h) \|\uvh\|_{1,T}^2.
\]

However, the practical manifestation of this $\bigO(h)$ limit depends heavily on the physical smoothness of the targeted eigenstate.
For the \textbf{fundamental state} and the first few excited states, the physical wavefunctions oscillate on a macroscopic scale $\lambda_{\ell}$ that is much larger than the mesh size ($\lambda_{\ell} \gg h$). The discrete gradient $\vtau_h$ approximates a highly smooth continuous field. The physical spatial derivatives scale as $||\vtau_h||_{H^{k+1}(T)} \sim (\lambda_{\ell})^{-m} \|\vtau_h\|_T$, which is practically bounded independently of $h$, and the scheme exhibits apparent optimal super-convergence $\bigO(h^{2k+2})$ on the eigenvalues, effectively bypassing the local $\bigO(h)$ aliasing limit.
For \textbf{highly excited states}, the physical wavelength approaches the grid resolution ($\lambda_{\ell} \approx h$). The eigenfunctions oscillate at the mesh scale, meaning the inverse inequality bound is physically saturated: the spatial derivatives genuinely scale as $h_T^{-m}$.
In this regime, the polynomial space is unable to adequately resolve the combined oscillation of the phase $e^{-i\chi}$ and the state. The discrete gauge error strictly degrades to $\bigO(h)$, requiring a fine mesh to recover asymptotic coherence.
This highlights that strong gauge invariance is lost at the algebraic level due to polynomial aliasing, but remains asymptotically optimal for macroscopic low-energy physical observables.
\end{remark}

\begin{corollary}[Consistency with continuous potential]
\label{cor:GT.chi_zero}
Let $\GT$ be the discrete gradient constructed with the projected potential $\vAT=\projT(\vA)$. For any $\uvTF \in \UT$, we have
\[
  (\GT \uvTF, \vtau)_T = (\GT[\vA] \uvTF, \vtau)_T + \bigO(h_T^{k+1}) \|\uvTF\|_{1,T} \|\vtau\|_T, \qquad \forall \vtau\in\Poly{k}(T)^d.
\]
\end{corollary}

\begin{proof} 
Applying Theorem \ref{thm:gauge_covariance} with the trivial gauge phase $\chi = 0$ gives $\uvTF[0] = \uvTF$ and $\vAT[0] = \vAT$. 
Furthermore, since $e^{-i0}\vtau = \vtau \in \Poly{k}(T)^d$, its $L^2$-projection is exact, meaning the projection error vanishes identically ($\vE = 0$). The total remainder $\Rgauge(\uvTF, \vtau)$ reduces solely to the potential approximation term $\term[2] = (v_T, (\vA - \vAT) \cdot \vtau)_T$. By standard polynomial approximation, this is bounded by $C h_T^{k+1} |\vA|_{W^{k+1,\infty}(T)} \|v_T\|_T \|\vtau\|_T$ without requiring any inverse inequality on the test function, yielding the optimal $L^2$ bound.
\end{proof}

\subsection{Discrete G\aa rding Inequality}

Before addressing the discrete G\aa rding Inequality, we need the following Lemma.

\begin{lemma} There exists a constant $C>0$, independent of $h$, such that the magnetic-free discrete covariante gradient $\GT[\vzero]$ satisfies
  \[
    || \GT[\vzero] \uuTF ||_T^2 +  s_T(\uuTF, \uuTF) \ge C ||\uuTF ||_{1,T}^2,
  \]
  for any $\uuTF \in \UT$.
  \label{lem:GT.zero.coercice}
\end{lemma}
\begin{proof} taking $\vtau = \GRAD w$ with $w\in\Poly{k+1}(T)$ and $\vA=\vzero$ in the definition of $\GT[\vA]$ \eqref{eq:def.GT.adjoint} yields
    \begin{align*}
      (\GT[\vzero] \uuTF, \GRAD w)_T &= -i\Big( (\GRAD u, \GRAD w)_T + \sumF (u_T - u_F, \GRAD w\cdot \vn)_F\Big) \\
                                     &= (-i\GRAD \pT \uuTF, \GRAD w)_T.
    \end{align*}
    This indicates that $-i\GRAD \pT \uuTF$ is the $L^2$ projection of $\GT[\vzero]\uuTF$ on the subspace $\GRAD \Poly{k+1}(T) \subset \Poly{k}(T)^d$.
    Hence,
    \[
      ||\GT[\vzero] \uuTF||^2_T \ge ||-i\GRAD \pT \uuTF||^2 \ge || \GRAD \pT \uuTF||.
    \]
    Adding $s_T(\uuTF,\uuTF)$ on both sides and using ~\cref{prop:pT.coercive} yields
      \begin{align*}
        ||\GT[\vzero] \uuTF||^2_T + s_T(\uuTF, \uuTF) &\ge C||\GRAD u_T||_T^2 + s_T(\uuTF, \uuTF), \\
                                                      &\ge \min(1,C)||\uuTF||_{1,T}^2.
      \end{align*}
\end{proof}

\begin{theorem}[Discrete Gauge-Invariant G\aa rding Inequality]
\label{thm:garding}
There exist constants $\alpha > 0$ and $\gamma > 0$, independent of $h$, such that for all $\uvh \in \Uh$:
\[
  a_h(\uvh, \uvh; \vA) \ge \alpha \|\uvh\|_{1,h}^2 - \gamma \left( \|\vA\|_{L^\infty(\DOM)}^2 + \|V\|_{L^\infty(\DOM)} \right) \|\uvh\|_{L^2(\DOM)}^2,
\]
where $\|\uvh\|_{L^2(\DOM)}^2 := \sumT \|v_T\|_T^2$.
\end{theorem}

\begin{proof} By the definition of the full covariant gradient $\GT[\vA]$:
\begin{align*}
  (\GT[\vA] \uvTF, \vtau)_T &= ((-i\GRAD - \vA)v_T, \vtau)_T - i \sumF \langle v_T - v_F, \vtau \cdot \vn \rangle_F \\
  &= (\GT[0] \uvTF, \vtau)_T - (\vA v_T, \vtau)_T.
\end{align*}

Since this equality holds for any test polynomial $\vtau \in \Poly{k}(T)^d$, it yields the exact algebraic identity:
\[
  \GT[\vA] \uvTF = \GT[\vzero] \uvTF - \projT (\vA v_T).
\]
Taking the squared $L^2$-norm of $\GT[\vA] \uvTF$ and using the elementary inequality $\|a - b\|^2 \ge \frac{1}{2}\|a\|^2 - \|b\|^2$, we get:
\[
  \|\GT[\vA] \uvTF\|_T^2 \ge \frac{1}{2}\|\GT[\vzero] \uvTF\|_T^2 - \|\projT(\vA v_T)\|_T^2.
\]
Since the $L^2$-orthogonal projector $\projT$ is a contraction (i.e., $\|\projT \vtau \|_T \le \|\vtau \|_T$), we can bound the magnetic perturbation as:
\[
  \|\projT(\vA v_T)\|_T^2 \le \|\vA v_T\|_T^2 \le \|\vA\|_{L^\infty(T)}^2 \|v_T\|_T^2.
\]
Adding the stabilization term $s_T(\uvTF, \uvTF)$ to both sides (and noting that $\frac{1}{2} s_T \le s_T$), we obtain:
\[
  \|\GT[\vA] \uvTF\|_T^2 + s_T(\uvTF, \uvTF) \ge \frac{1}{2} \left( \|\GT[\vzero] \uvTF\|_T^2 + s_T(\uvTF, \uvTF) \right) - \|\vA\|_{L^\infty(T)}^2 \|v_T\|_T^2.
\]
Using the upper bound in Lemma \ref{lem:GT.zero.coercice} yields:
\[
  \|\GT[\vA] \uvTF\|_T^2 + s_T(\uvTF, \uvTF) \ge \frac{C}{2} \|\uvTF\|_{1,T}^2 - \|\vA\|_{L^\infty(T)}^2 \|v_T\|_T^2.
\]
The local bilinear form is given by $a_T(\uvTF, \uvTF) = \|\GT[\vA] \uvTF\|_T^2 + s_T(\uvTF, \uvTF) + (V \pi_T^k v_T, v_T)_T$.
For the scalar potential term, the contraction property of $\pi_T^k$ yields:
\[
  (V \pi_T^k v_T, v_T)_T \ge -\|V\|_{L^\infty(T)} \|\pi_T^k v_T\|_T \|v_T\|_T \ge -\|V\|_{L^\infty(T)} \|v_T\|_T^2.
\]
Summing over all elements $T \in \Th$, we finally obtain:
{\small
\begin{align*}
a_h(\uvTF, \uvTF) &= \sumT a_T(\uvTF, \uvTF)  \ge \sumT \left( \frac{C}{2} \|\uvTF\|_{1,T}^2 - \left(\|\vA\|_{L^\infty(T)}^2 + \|V\|_{L^\infty(T)}\right) \|v_T\|_T^2 \right) \\
  &\ge \frac{C}{2} \|\uvTF\|_{1,h}^2 - \left(\|\vA\|_{L^\infty(\DOM)}^2 + \|V\|_{L^\infty(\DOM)}\right) \|\uvh\|_{L^2(\DOM)}^2.
\end{align*}
}
Setting $\alpha = \frac{C}{2}$ and $\gamma = 1$ concludes the proof.
\end{proof}

\subsection{A Priori Error Estimate for the Stationary Schr\"{o}edinger problem}
\label{sec:a_priori_error}

In this section, we focus on the stationary Schr\"{o}dinger problem:
find $u\in H^1_0(\DOM; \Complex)$, such that
\begin{equation}
  \label{eq:continuous.stationnary}
  \cG^*\cG u  + Vu = f, \qquad \text{ in } \DOM.
\end{equation}
In particular, we establish a priori error estimates for the discrete equivalent of the above problem : find $\uuh \in \Uh$ such that 
\begin{align}
  \label{eq:discrete.stationnary}
  a_h(\uuh, \uvh; \vA[h]) &= \ell(\uvh), \qquad \forall \uvh \in \Uh,
\end{align}
where $a_h(\cdot, \cdot; \vA[h])$ is defined using $a_T$ defined in \eqref{eq:aT.defn} with the potential $V$, i.e.
\[
  a_h(\uuh, \uvh; \vA[h]) := \sumT \left\{a_T(\uuTF, \uvTF) + (Vu_T, v_T)_T \right\}.
\]
We also assume that the mesh is sufficiently fine so that this problem is well-posed.

\begin{theorem}[Optimal a priori error estimate]
\label{thm:a_priori}
Let $u \in H_0^1(\Omega;\Complex)$ be the exact solution of the continuous problem \eqref{eq:continuous.stationnary} and $\uuh \in U_h^k$ be the discrete HHO solution from \eqref{eq:discrete.stationnary}. Assuming sufficient regularity $u \in H^{k+2}(\Omega;\Complex)$, $\vA \in W^{k+1,\infty}(\Omega;\Real^d)$, and $V \in L^\infty(\Omega;\Real)$, there exists a constant $C > 0$, independent of $h$, $\vA$, and $V$, such that for $h$ sufficiently small:
\begin{equation}
  \|\uuh - \Ih u \|_{1,h} \le C h^{k+1} \left( \|u\|_{H^{k+2}(\Omega)} + \big( \|\vA\|_{W^{k+1,\infty}(\Omega)} + \|V\|_{L^\infty(\Omega)} \big) \|u\|_{H^{k+1}(\Omega)} \right).
\end{equation}
Moreover, an $L^2$-error bound for the cell unknowns can be given as
\begin{equation}
  \| \uuh  - \projh u\|_{L^2(\Omega)} \le C h^{k+2} \left( \|u\|_{H^{k+2}(\Omega)} + \big( \|\vA\|_{W^{k+1,\infty}(\Omega)} + \|V\|_{L^\infty(\Omega)} \big) \|u\|_{H^{k+1}(\Omega)} \right).
\end{equation}
\end{theorem}
\begin{proof}[Proof of Theorem \ref{thm:a_priori}]

First, we observe that there exists a constant $\beta > 0$, independent of $h$, such that for all $\uwh \in \Uh$:
$$\beta \|\uwh\|_{1,h} \le \sup_{\uvh \in U_h^k \setminus \{0\}} \frac{|a_h(\uwh, \uvh)|}{\|\uvh\|_{1,h}}.$$
Indeed, the above discrete inf-sup condition follows from the discrete G\aa rding inequality via a standard compactness argument in HHO spaces (see \cite[Theorem 6.41]{DiPietro2020} with $p=2$).
Let $\ueh := \uuh - \Ih u \in \Uh$ be the discrete error.
Applying the inf-sup condition to $\ueh$ yields:
$$ \beta \|\ueh\|_{1,h} \le \sup_{\uvh \in U_h^k \setminus \{0\}} \frac{|a_h(\uuh, \uvh) - a_h(\Ih u, \uvh)|}{\|\uvh\|_{1,h}}.$$
Since $\uuh$ is the discrete solution, it satisfies $a_h(\uuh, \uvh) = \ell(\uvh)$.
We define the consistency error functional as $\mathcal{E}_h(\uvh) := \ell(\uvh) - a_h(\Ih u, \uvh)$, for any $\uvh \in \Uh$.
Since $u$ solves the exact continuous equation \eqref{eq:continuous.stationnary}, we have the following splitting

\[
  \Eh(\uvh) = (f, \vh) - a_h(\Ih u, \uvh) =: \term[grad] + \term[stab] + \term[pot],
\]
where $\term[pot] :=- \sumT (V\projT u, v_T)_T$, $\term[stab] := -  \sumT s_T(\IT u, \uvTF)$ and 
\[
  \term[grad] := \sumT (\cG^*\cG u + Vu, v_T)_T - (\GT \IT u, \GT \uvTF)_T.
\]

\textbf{Step 1.} We first bound the term $\term[pot]$, by Cauchy-Schwarz, the definition of $h$, the local approximation property of $\projT$ in \eqref{eq:proj.approx.cell}, the global discrete Poincaré inequality in \cref{lem:discrete.Poincare} for $\uvTF\in\UT$ and the boundedness of the potential $V$ yields
\begin{align*}
  |\term[pot]| &= \left|\sumT (Vu - V\projT u, \vT)_T\right| \le \sumT ||V||_{L^\infty(\DOM)}\|u-\projT u\|_T \|\vT\|_T \\
               &\le ||V||_{L^\infty(\DOM)}\sum_T C h_T^{k+1}|u|_{H^{k+2}(T)} \| \vT \|_T \\
               &= ||V||_{L^\infty(\DOM)} C h^{k+1}|u|_{H^{k+2}(\DOM)} \| \vh \|_{L^2(\DOM)}  \\
               &\lesssim ||V||_{L^\infty(\DOM)} h^{k+1} |u|_{H^{k+2}(\DOM)} \| \uvh \|_{1,h}.
\end{align*}

\textbf{Step 2.} Then, for $\term[stab]$, we use Cauchy-Schwarz and \cref{prop:sT.bound} 
\begin{align*}
  |\term[stab]| &= \left| \sumT s_T(\IT u, \uvTF)\right| = \left| \sumT \sum_{F\in\FT} h_F^{-1} (\delT[TF](u) - \delT(u), v_F - v_T)_F\right|\\
                &\le   \sumT \sum_{F\in\FT} h_F^{-\tfrac12} \|\delT[TF](u) - \delT(u)\|_F h_F^{-\tfrac12}\|v_F - v_T\|_F \\
                &\le \sumT \left( \sumF h_F^{-1} \|\delT[TF](u) - \delT(u)\|_F^2 \right)^{\tfrac12} \left( \sum_{F\in\FT} h_F^{-1}\|v_F - v_T\|_F^2  \right)^{\tfrac12}\\
                &\le \sumT s_T(\IT u, \IT u)^{\tfrac12}s_T(\uvTF, \uvTF)^{\tfrac12} \\
                & \le C h^{k+1} \sumT \|u\|_{H^{k+2}(T)} \|\uvTF\|_{1,T}  \lesssim h^{k+1} |u|_{H^{k+2}(\DOM)} \|\uvh\|_{1,h} 
\end{align*}

\textbf{Step 3.} The last term $\term[grad]=\sumT M_T$ reads as a sum of three terms 
\[
  \small 
M_T = (\cG u, \cG v_T)_T - i\sumF(\cG u \cdot \vn, v_T - v_F)_F  - (\GT\IT u, \GT \uvTF)_T.
\]
We first remark that, since $u$ is a regular function, we have $\projF(u)= \projT(u)_{| F}$ for any $F\in \Fh$, hence
\[
  (\GT[\vA] \IT u, \vtau)_T = (\cG u, \vtau)_T - i\sumF (\projF(u) - \projT(u), \vtau\cdot \vn)_F = (\cG u, \vtau)_T,
\]
for any $\vtau \in \Poly{k}(T)^d$, in other terms
\begin{equation}
  \label{eq:GT.VA.projT}
\GT[\vA] \IT u = \projT(\cG u), \quad \forall \vA \in W^{k+1,\infty}(\DOM).
\end{equation}
Replacing $\cG u$ with $\cG u - \projT(\cG u) + \projT (\cG u)$ in the volumetric term \textit{and} the face terms, yields
{\small
\begin{align*}
  M_T &= (\cG u, \cG v_T)_T - i\sumF(\cG u \cdot \vn, v_T - v_F)_F  - (\GT\IT u, \GT \uvTF)_T \\
      &= (\projT(\cG u), \cG v_T)_T + \mhl{(\cG u - \projT(\cG u), \cG v_T)_T} - i\sumF(\projT(\cG u) \cdot \vn, v_T - v_F)_F  \\
      & \qquad - (\GT\IT u, \GT \uvTF)_T \mhl{- i\sumF((\cG u - \projT(\cG u)) \cdot \vn, v_T - v_F)_F} \\
      &= (\projT(\cG u), \cG v_T)_T -i\sumF(\projT(\cG u) \cdot \vn, v_T - v_F)_F + \mhl{\term[2]} \\
      & \qquad - (\GT\IT u, \GT \uvTF)_T
\end{align*}
}
where we have defined $\mhl{\term[2]}$ as the sum of the two terms in grey in the second equality.
We notice that, by definition of $\GT[\vA]$ \eqref{eq:def.GT.adjoint} with $\vtau=\projT(\cG u)$ as a test function, we have 
\[
  (\projT(\cG u), \GT[\vA] \uvTF)_T = (\projT(\cG u), \cG v_T)_T -i\sumF(\projT(\cG u) \cdot \vn, v_T - v_F)_F,
\]
which corresponds exactly to the first two terms, i.e. $M_T$ rewrites as
\[
  M_T = (\projT(\cG u), \GT[\vA] \uvTF)_T - (\GT\IT u, \GT \uvTF)_T + \term[2].
\]
Using \eqref{eq:GT.VA.projT} with $\vA=\vAT$, we have $\GT \IT u = \projT(\cGAT u)$, therefore $M_T$ rewrites as
\begin{align*}
  M_T &= (\projT(\cG u), \GT[\vA] \uvTF)_T- (\projT(\cGAT u), \GT \uvTF)_T + \term[2]\\
      &= (\projT(\cG u), \GT[\vA] \uvTF)_T \pm (\projT(\cG u), \GT[\vAT] \uvTF) - (\projT(\cGAT u), \GT \uvTF)_T  + \term[2]\\
      &= (\projT(\cG u), \GT[\vA] \uvTF -  \GT[\vAT] \uvTF)_T +  (\projT(\cG u) - \projT(\cGAT u), \GT \uvTF)_T + \term[2] \\
      &=: \term[1] + \term[2].
\end{align*}
It remains to bound $\term[1]$ and $\term[2]$. 
For the first term, we have
\begin{align*}
  |\term[1]| &\le ||\projT (\cG u)||_T ||\GT[\vA] \uvTF - \GT \uvTF||_T + ||\projT(\cG u - \cG[\vAT] u)||_T ||\GT \uvTF||_T \\
             &\le ||\cG u||_T Ch_T^{k+1} + ||(\vAT - \vA)u||_T ||\uvTF||_{1,T} \\
             &\le C||\cG u||_T h_T^{k+1} + ||\vA - \vAT||_{L^{\infty}} ||u||_T ||\uvTF||_{1,T} \\
             &\le C||\cG u||_T h_T^{k+1} + \widetilde{C}h^{k+1}|\vA|_{W^{k+1,\infty(T)}} ||u||_T ||\uvTF||_{1,T} \\
             &\lesssim h^{k+1}|\vA|_{W^{k+1,\infty}} |u|_{H^{k+1}(T)} ||\uvTF||_{1,T}
\end{align*}
where we have used Corollary \ref{cor:GT.chi_zero}, the definition of $\cG$ together with the linearity of $\projT$, and the boundedness of $\projT$ and $\GT$.
For the second term $\term[2]$, we remark that
\begin{align*}
  \term[2] &= (\cG u - \projT(\cG u), \cG v_T)_T  - i\sumF((\cG u - \projT(\cG u)) \cdot \vn, v_T - v_F)_F \\
           &= (\cG u - \projT(\cG u), -i\GRAD v_T - \vA v_T)_T - i\sumF((\cG u - \projT(\cG u)) \cdot \vn, v_T - v_F)_F \\
           &= (\cG u - \projT(\cG u), - \vA v_T)_T - i\sumF((\cG u - \projT(\cG u)) \cdot \vn, v_T - v_F)_F
\end{align*}
since $\GRAD v_T \in \Poly{k-1}(T)^d$, and $\cG u -\projT(\cG u)$ cancels against any polynomial in $\Poly{k}(T)^d$. 
We can now estimate $\term[2]$ as follows:
\begin{align*}
  |\term[2]| & \le ||\cG u - \projT(\cG u)||_T ||\vA||_{L^{\infty}(T)}||v_T||_T + \sum_F ||\cG u - \projT(\cG u)||_F ||v_T - v_F||_F \\
             & \le C ||\vA||_{L^{\infty}(T)} h_T^{k+1}|\cG u|_{H^{k+1}(T)}  ||v_T||_T + \sum_F Ch_T^{k+1}|\cG u|_{H^{k+1}(T)} h^{-\frac12}||v_T - v_F||_F \\
             & \le C ||\vA||_{L^{\infty}(T)} h_T^{k+1}|\cG u|_{H^{k+1}(T)}  ||v_T||_T + \Np C h_T^{k+1} |\cG u|_{H^{k+1}(T)} ||\uvTF||_{1,T}.
\end{align*}
where we have used \eqref{eq:proj.approx.cell} for the first term and \eqref{eq:proj.approx.trace} for the trace terms, and the definition of the norms \eqref{eq:discrete.norms}.
Summing $\term[1]$ and $\term[2]$ over all elements $T \in \Th$, using
\[
|\cG u|_{H^{k+1}(T)} \le C \left( |u|_{H^{k+2}(T)} + \|\mathbf{A}\|_{W^{k+1,\infty}(T)} \|u\|_{H^{k+1}(T)} \right),
\]
and applying the discrete Cauchy-Schwarz and Poincaré inequalities (Lemma ~\ref{lem:discrete.Poincare}) yields the global optimal bound for $|\term[grad]|$, which concludes the proof for the energy norm.
The $L^2$-error bound relies on the classic Aubin-Nitsche duality argument.
\end{proof}

\section{Numerical Results}
\label{sec:numerics}

\begin{figure}
  \centering
  \includegraphics[width=0.6\textwidth]{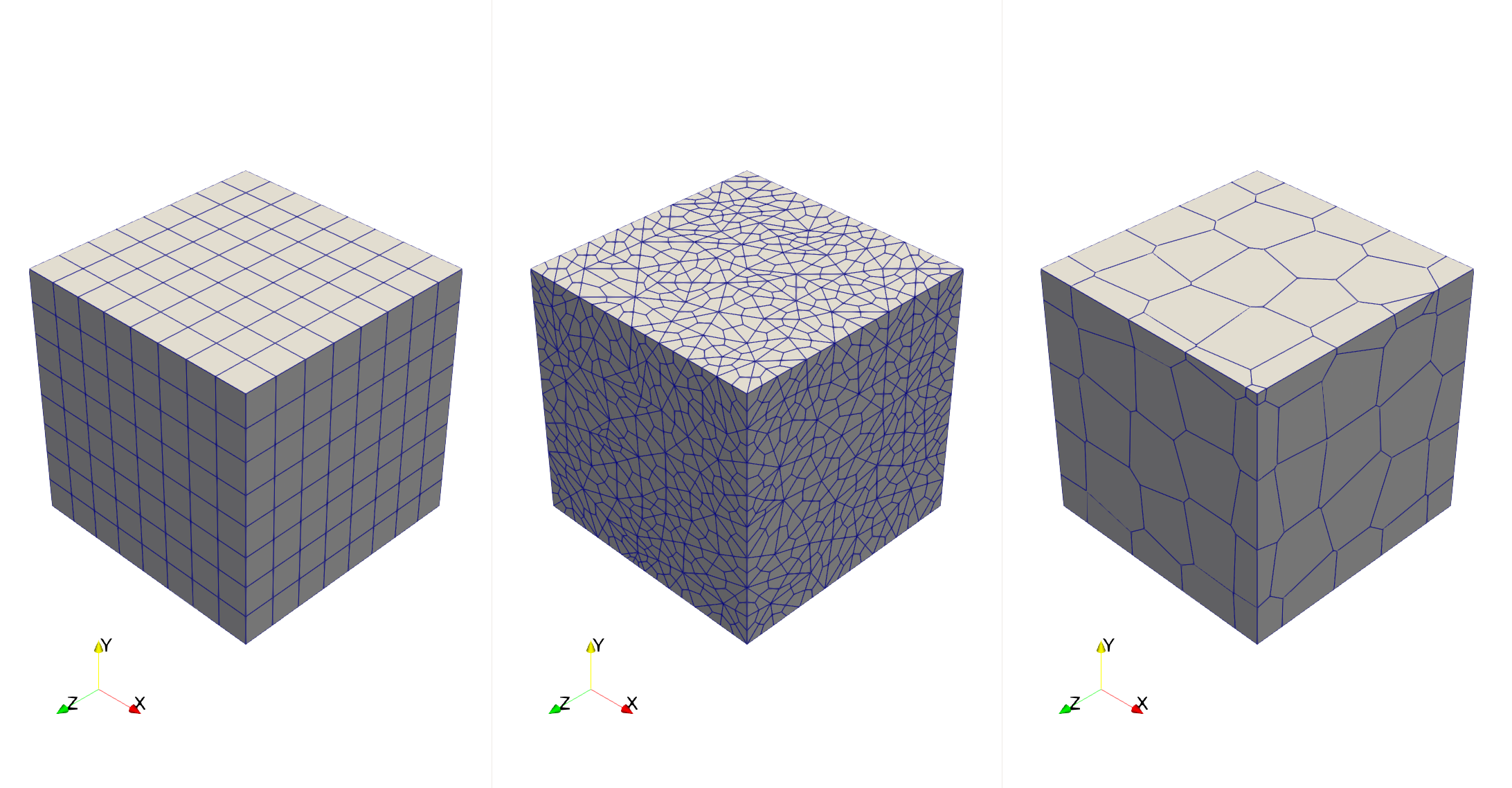}
  \includegraphics[width=0.35\textwidth]{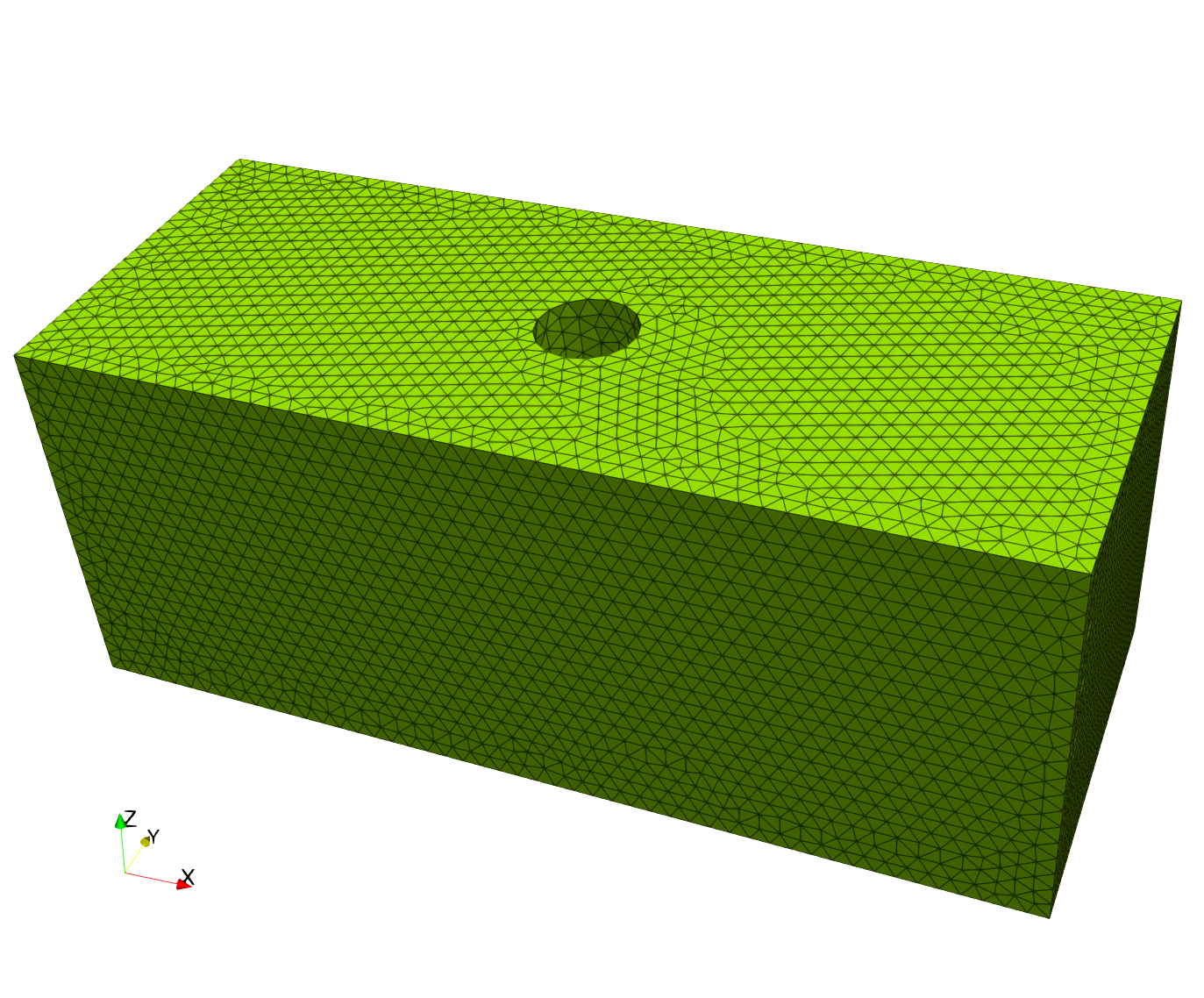}
  \caption{Meshes used in the numerical experiments in \Cref{sec:numerics} -- From first to third: the 8x8x8 cubes, \texttt{hex} and the Voronoi meshes;  the fourth is a 3D simplicial mesh utilized for the Aharonov-Bohm experimental test case ($\approx$ 88k cells).}
  \label{fig:meshes}
\end{figure}
In this section, we validate our results through two distinct test cases.
First, we address the eigenvalue problem in Eq. \cref{eq:Schrodinger.stationnary} and confirm that the discrepancies between the eigenvalues are effectively controlled by the upper bound prediction in Theorem \ref{thm:gauge_covariance}.
Second, we investigate the unsteady Schr\"{o}dinger equation in Eq. \eqref{eq:Schrodinger.time} to replicate the Aharonov-Bohm experiment, wherein the presence of a magnetic field induces a phase shift in the propagating photons.
Both simulations were conducted using the \texttt{HArDCore3D} library\footnote{\texttt{https://github.com/jdroniou/HArDCore3D-release}}, a three-dimensional implementation of the Hybrid High-Order (HHO) schemes \cite{DiPietro2020}, in conjunction with the \texttt{Spectra} library \cite{spectralib} for eigenvalue computations.

\subsection{The Eigenvalue Problem: Fock-Darwin Spectrum and Gauge Invariance}
\label{sec:num.eig.problem}

To assess the accuracy of our HHO scheme for the eigenvalue problem \eqref{eq:Schrodinger.stationnary} and numerically validate the discrete gauge covariance established in Theorem~\ref{thm:gauge_covariance}, we consider a two-dimensional quantum harmonic oscillator subjected to a uniform perpendicular magnetic field $\vec{B} = B \vec{e}_z$. 

We operate on a sufficiently large truncated domain $\Omega = [-L, L]^2$ with homogeneous Dirichlet boundary conditions. Following our initial non-dimensionalization ($\hbar = m = q = 1$), the scalar confinement potential is given by:
\begin{equation}
    V(x,y) = \frac{1}{2} \omega_0^2 (x^2 + y^2),
\end{equation}
where $\omega_0 > 0$ is the oscillator frequency.
A constant magnetic field $\vB$ can be represented by infinitely many vector potentials $\vA$ related by gauge transformations \eqref{def:continuous_gauge_transformation}.
To highlight the exact discrete gauge invariance of our scheme under discrete gauge transformations, we solve the generalized eigenvalue problem using three distinct gauges, the symmetric gauge $\vA_{\text{sym}}$, the Landau $\vA_{\text{Lan}}$ and a manufactured gauge $\vA_{\text{smooth}}$ defined as 

\[
  \vA_{\text{sym}} = \frac{B}{2} \begin{pmatrix} -y \\ x \end{pmatrix},
  \qquad
  \vA_{\text{Lan}} = \begin{pmatrix} -By \\ 0 \end{pmatrix},
  \quad
  \vA_{\text{smooth}} =  \begin{pmatrix} -\tfrac12 B y + 0.1 \\ \tfrac12 B x + 0.1 \end{pmatrix}.
\]
Note that these quantities are related by $\vA_{\text{Lan}} = \vA_{\text{sym}} + \GRAD \left(-\frac{B}{2} xy\right)$ and $\vA_{\text{smooth}} = \vA_{\text{sym}} + \GRAD \left(0.1(x+y)\right)$.
The exact analytical eigenvalues of \eqref{eq:Schrodinger.stationnary} in three dimensions are derived from the 2D Fock-Darwin spectrum \cite[Eq. (23)]{Governale1998}, extended by a 1D particle-in-a-box spectrum along the $z$-axis due to the homogeneous Dirichlet boundary conditions imposed on $z \in [-L,L]$. They are independent of the gauge $\vA$ and are given, for any quantum numbers $n \in \mathbb{N}$, $m \in \mathbb{Z}$, and $n_z \in \mathbb{N}^*$, by:
\begin{equation}
  E_{n,m,n_z} = \sqrt{B^2 + 2\omega_0^2} (2n + |m| + 1) - m B + \frac{n_z^2\pi^2}{(2L)^2}.
\end{equation}
The fundamental energy level, known as the Fock-Darwin ground state, corresponds to $n=m=0$ and $n_z=1$, yielding:
\begin{equation}
  E_{0,0,1} = \sqrt{B^2 + 2\omega_0^2} + \frac{\pi^2}{(2L)^2}.
  \label{eq:fundamental.energy}
\end{equation}
In what follows, we assume that $w_0=1$ and $B=1$.
The discrete eigenvalue problem reads : find $\uh\in\Uh$ and $\lambda_h\in\Real$ such that $a_h(\uuh, \uvh; \vA_T)=\lambda_h (\uuh, \uvh)_{L^2(\DOM)}$. 
In Figure \ref{fig:eig.gauge.invariance}, we report the maximum deviation absolute error
\[
  \mathtt{dev\_error}(\mathrm{g}) := \max_{0\le j \le N-1} |\lambda_{h, j}^{(\text{sym})} - \lambda_{h, j}^{(\text{gauge})}|,
\]
with $\mathrm{g}\in \{ \mathrm{smooth}, \mathrm{Landau}\}$, for the first $N=5$ eigenvalues across polynomial degrees from $k=0$ to $k=3$, on progressively refined meshes.
We observe a superconvergence of one order for $k=0$ (in blue), but an optimal decay rate for $k=2$ (in black), suggesting that the prediction rate $\bigO(h^{k+1})$ is optimal and confirming that the discrete spectrum behaves covariantly consistent with the discretization error.
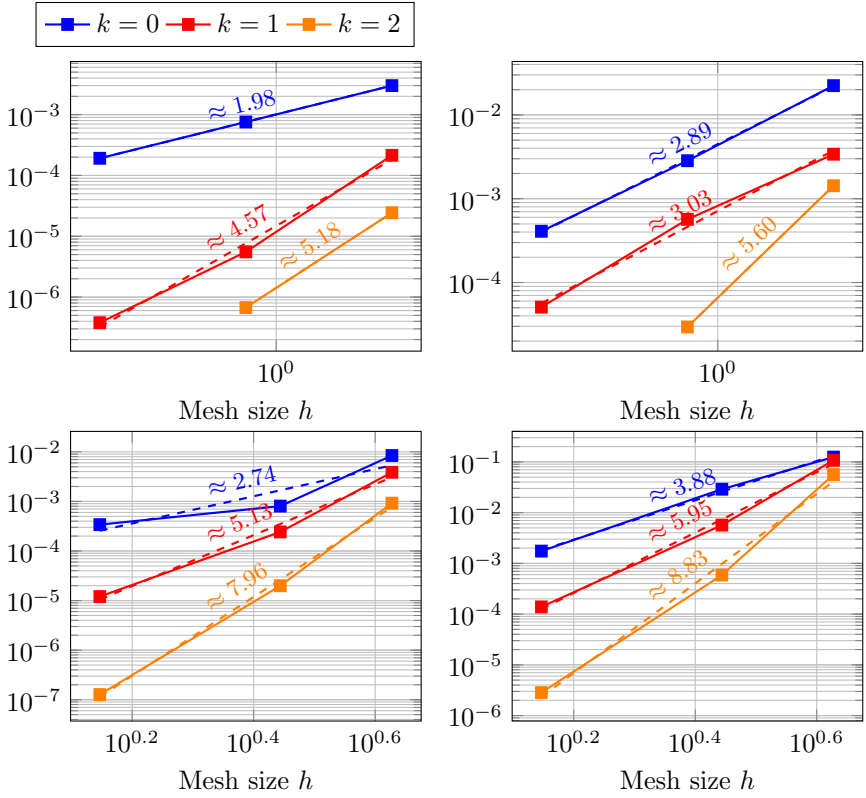
\begin{figure}
  \begin{tikzpicture}
  \begin{loglogaxis}[
      xlabel={Mesh size $h$},
      legend style={at={(-0.1,1.05)},anchor=south west},
      legend columns=3, 
      grid=both,
      width=0.48\textwidth
  ]
  \addplot[blue, thick, mark=square*] coordinates {
      (1.732051e+00, 3.000568e-03)
      (8.660254e-01, 7.570668e-04)
      (4.330127e-01, 1.915630e-04)
  };
  \addlegendentry{$k=0$}
  \addplot[blue, dashed, thick, forget plot] coordinates {
      (4.330127e-01, 1.914714e-04)
      (1.732051e+00, 2.999132e-03)
  } node[pos=0.5, above, sloped, font=\small] {$\approx 1.98$};
  \addplot[red, thick, mark=square*] coordinates {
      (1.732051e+00, 2.142390e-04)
      (8.660254e-01, 5.537964e-06)
      (4.330127e-01, 3.792976e-07)
  };
  \addlegendentry{$k=1$}
  \addplot[red, dashed, thick, forget plot] coordinates {
      (4.330127e-01, 3.224411e-07)
      (1.732051e+00, 1.821247e-04)
  } node[pos=0.5, above, sloped, font=\small] {$\approx 4.57$};
  \addplot[orange, thick, mark=square*] coordinates {
      (1.732051e+00, 2.437371e-05)
      (8.660254e-01, 6.725295e-07)
  };
  \addlegendentry{$k=2$}
  \addplot[orange, dashed, thick, forget plot] coordinates {
      (8.660254e-01, 6.725295e-07)
      (1.732051e+00, 2.437371e-05)
  } node[pos=0.5, above, sloped, font=\small] {$\approx 5.18$};
  \end{loglogaxis}
\end{tikzpicture}
  \begin{tikzpicture}
  \begin{loglogaxis}[
      xlabel={Mesh size $h$},
      legend pos=south east,
      grid=both,
      width=0.48\textwidth
  ]
  \addplot[blue, thick, mark=square*] coordinates {
      (1.732051e+00, 2.235178e-02)
      (8.660254e-01, 2.840135e-03)
      (4.330127e-01, 4.084866e-04)
  };
  \addplot[blue, dashed, thick, forget plot] coordinates {
      (4.330127e-01, 4.001374e-04)
      (1.732051e+00, 2.189492e-02)
  } node[pos=0.5, above, sloped, font=\small] {$\approx 2.89$};
  \addplot[red, thick, mark=square*] coordinates {
      (1.732051e+00, 3.395875e-03)
      (8.660254e-01, 5.669889e-04)
      (4.330127e-01, 5.086607e-05)
  };
  \addplot[red, dashed, thick, forget plot] coordinates {
      (4.330127e-01, 5.641436e-05)
      (1.732051e+00, 3.766285e-03)
  } node[pos=0.5, above, sloped, font=\small] {$\approx 3.03$};
  \addplot[orange, thick, mark=square*] coordinates {
      (1.732051e+00, 1.429109e-03)
      (8.660254e-01, 2.946901e-05)
  };
  \addplot[orange, dashed, thick, forget plot] coordinates {
      (8.660254e-01, 2.946901e-05)
      (1.732051e+00, 1.429109e-03)
  } node[pos=0.5, above, sloped, font=\small] {$\approx 5.60$};
  \end{loglogaxis}
\end{tikzpicture}\\
  \begin{tikzpicture}
  \begin{loglogaxis}[
      xlabel={Mesh size $h$},
      grid=both,
      width=0.48\textwidth
  ]
  \addplot[blue, thick, mark=square*] coordinates {
      (4.242641e+00, 8.423811e-03)
      (2.779004e+00, 8.052621e-04)
      (1.401890e+00, 3.401612e-04)
  };
  \addplot[blue, dashed, thick, forget plot] coordinates {
      (1.401890e+00, 2.569741e-04)
      (4.242641e+00, 5.352151e-03)
  } node[pos=0.5, above, sloped, font=\small] {$\approx 2.74$};
  \addplot[red, thick, mark=square*] coordinates {
      (4.242641e+00, 3.873181e-03)
      (2.779004e+00, 2.424013e-04)
      (1.401890e+00, 1.205513e-05)
  };
  \addplot[red, dashed, thick, forget plot] coordinates {
      (1.401890e+00, 1.046462e-05)
      (4.242641e+00, 3.080966e-03)
  } node[pos=0.5, above, sloped, font=\small] {$\approx 5.13$};
  \addplot[orange, thick, mark=square*] coordinates {
      (4.242641e+00, 9.223936e-04)
      (2.779004e+00, 1.992894e-05)
      (1.401890e+00, 1.273458e-07)
  };
  \addplot[orange, dashed, thick, forget plot] coordinates {
      (1.401890e+00, 1.141051e-07)
      (4.242641e+00, 7.723316e-04)
  } node[pos=0.5, above, sloped, font=\small] {$\approx 7.96$};
  \end{loglogaxis}
\end{tikzpicture}
  \begin{tikzpicture}
  \begin{loglogaxis}[
      xlabel={Mesh size $h$},
      legend pos=south east,
      grid=both,
      width=0.48\textwidth
  ]
  \addplot[blue, thick, mark=square*] coordinates {
      (4.242641e+00, 1.245194e-01)
      (2.779004e+00, 2.894204e-02)
      (1.401890e+00, 1.746522e-03)
  };
  \addplot[blue, dashed, thick, forget plot] coordinates {
      (1.401890e+00, 1.822862e-03)
      (4.242641e+00, 1.334401e-01)
  } node[pos=0.5, above, sloped, font=\small] {$\approx 3.88$};
  \addplot[red, thick, mark=square*] coordinates {
      (4.242641e+00, 1.073880e-01)
      (2.779004e+00, 5.683175e-03)
      (1.401890e+00, 1.389544e-04)
  };
  \addplot[red, dashed, thick, forget plot] coordinates {
      (1.401890e+00, 1.257858e-04)
      (4.242641e+00, 9.141613e-02)
  } node[pos=0.5, above, sloped, font=\small] {$\approx 5.95$};
  \addplot[orange, thick, mark=square*] coordinates {
      (4.242641e+00, 5.632708e-02)
      (2.779004e+00, 5.889607e-04)
      (1.401890e+00, 2.838730e-06)
  };
  \addplot[orange, dashed, thick, forget plot] coordinates {
      (1.401890e+00, 2.335812e-06)
      (4.242641e+00, 4.109166e-02)
  } node[pos=0.5, above, sloped, font=\small] {$\approx 8.83$};
  \end{loglogaxis}
\end{tikzpicture}
  \caption{Gauge transformation errors : (Left) $\mathtt{dev\_error}(\text{smooth})$; (Right) $\mathtt{dev\_error}(\text{Landau})$, on three different meshes : cubes (Top) and  hex (Bottom).}
  \label{fig:eig.gauge.invariance}
\end{figure}
\begin{figure}[htbp]
    \centering
    \begin{tikzpicture}
  \begin{loglogaxis}[
      xlabel={Mesh size $h$},
      ylabel={Relative error $\mathtt{rel\_err}$},
      legend style={at={(-0.1,1.05)},anchor=south west},
      legend columns=3, 
      grid=both,
      width=0.48\textwidth
  ]
  \addplot[blue, thick, mark=*] coordinates {
      (3.464102e+00, 1.844458e-01)
      (1.732051e+00, 1.554502e-02)
      (8.660254e-01, 3.330111e-03)
      (4.330127e-01, 7.927928e-04)
  };
  \addlegendentry{$k=0$}
  \addplot[blue, dashed, thick, forget plot] coordinates {
      (4.330127e-01, 6.373539e-04)
      (3.464102e+00, 1.365081e-01)
  } node[pos=0.5, above, sloped, font=\small] {$\approx 2.58$};
  \addplot[red, thick, mark=*] coordinates {
      (3.464102e+00, 2.384874e-02)
      (1.732051e+00, 3.750477e-03)
      (8.660254e-01, 3.174437e-04)
      (4.330127e-01, 1.563601e-05)
  };
  \addlegendentry{$k=1$}
  \addplot[red, dashed, thick, forget plot] coordinates {
      (4.330127e-01, 2.081951e-05)
      (3.464102e+00, 3.200383e-02)
  } node[pos=0.5, above, sloped, font=\small] {$\approx 3.53$};
  \addplot[orange, thick, mark=*] coordinates {
      (3.464102e+00, 1.797918e-02)
      (1.732051e+00, 2.619107e-04)
      (8.660254e-01, 1.694339e-06)
  };
  \addlegendentry{$k=2$}
  \addplot[orange, dashed, thick, forget plot] coordinates {
      (8.660254e-01, 1.924e-06)
      (3.464102e+00, 2.073e-02)
  } node[pos=0.5, above, sloped, font=\small] {$\approx 6.69$};
  \end{loglogaxis}
\end{tikzpicture}
    \begin{tikzpicture}
  \begin{loglogaxis}[
      xlabel={Mesh size $h$},
      grid=both,
      width=0.48\textwidth
  ]
  \addplot[blue, thick, mark=*] coordinates {
      (6.612884e+00, 1.001138e-01)
      (3.632992e+00, 3.424090e-02)
      (2.442501e+00, 4.631369e-02)
      (1.771054e+00, 1.700776e-02)
      (9.059107e-01, 5.068913e-03)
  };
  \addplot[blue, dashed, thick, forget plot] coordinates {
      (9.059107e-01, 6.297735e-03)
      (6.612884e+00, 1.091573e-01)
  } node[pos=0.5, above, sloped, font=\small] {$\approx 1.44$};
  \addplot[red, thick, mark=*] coordinates {
      (6.612884e+00, 1.236379e-01)
      (3.632992e+00, 2.411750e-02)
      (2.442501e+00, 5.055442e-04)
      (1.771054e+00, 9.092966e-04)
      (9.059107e-01, 8.791179e-05)
  };
  \addplot[red, dashed, thick, forget plot] coordinates {
      (9.059107e-01, 5.836872e-05)
      (6.612884e+00, 1.049066e-01)
  } node[pos=0.5, above, sloped, font=\small] {$\approx 3.77$};
  \addplot[orange, thick, mark=*] coordinates {
      (6.612884e+00, 2.455461e-02)
      (3.632992e+00, 2.956621e-03)
      (2.442501e+00, 2.368012e-04)
      (1.771054e+00, 1.567522e-04)
      (9.059107e-01, 4.223972e-06)
  };
  \addplot[orange, dashed, thick, forget plot] coordinates {
      (9.059107e-01, 5.135075e-06)
      (6.612884e+00, 2.881623e-02)
  } node[pos=0.6, below, sloped, font=\small] {$\approx 4.34$};
  \end{loglogaxis}
\end{tikzpicture}
    \caption{Convergence decay of the relative error \eqref{eq:rel_err} on the cubes (left) and Voronoi (right) meshes.}
    \label{fig:eig.convergence}
\end{figure}
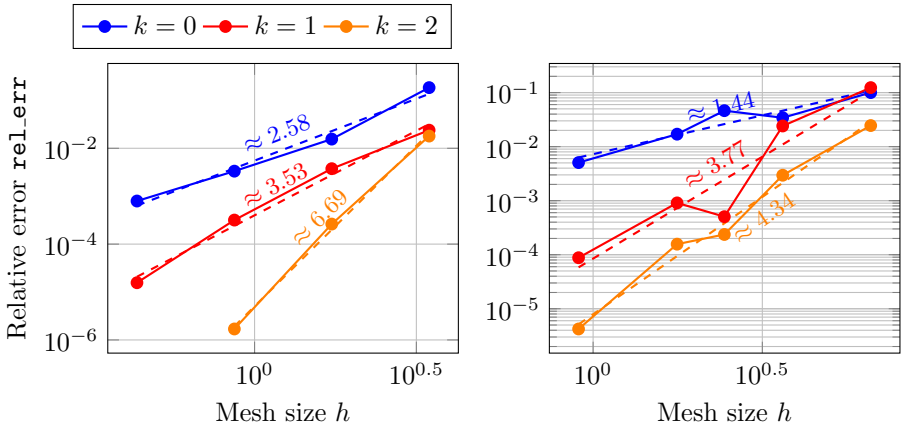
In our context, the fundamental energy \eqref{eq:fundamental.energy} is given as $E_{0,0,1} = \sqrt{3}+\frac{\pi^2}{64} \approx 1.8862633763358985.$
In Figure \ref{fig:eig.convergence}, we plot the relative errors of the fundamental energy level
\begin{equation}
  \mathtt{rel\_err} := |\lambda_{h, 0} - E_{0,0,1}| / E_{0,0,1},
  \label{eq:rel_err}
\end{equation}
against the mesh size $h$ on logarithmic scales, for two families of meshes : cubes and Voronoi meshes, see Fig. \ref{fig:meshes}.
We observe that the discrete eigenvalues converge towards the exact Fock-Darwin levels, the convergence is nearly at the optimal rate of $\bigO(h^{2k+2})$ for the cubes, which confirm the high-order accuracy of the proposed discrete covariant gradient for spectral computations.
On the other hand, suboptimal error rates are observed for the Voronoi mesh family, which might be due to boundary conditions together with higher value of $h$. 

\subsection{The Aharonov-Bohm Effect}
\label{sec:num.AB}

The Aharonov-Bohm (AB) effect is a quantum mechanical phenomenon in which a charged particle
is affected by the electromagnetic vector potential $\vA$ even in regions where the magnetic
field $\vB = \ROT\vA$ vanishes identically \cite{Aharonov1959}.

\begin{figure}
  \centering
  \includegraphics[width=0.8\textwidth]{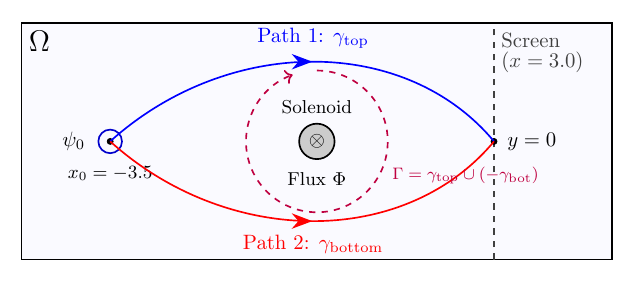}
  \caption{Domain, settings and paths in the the Aharonov-Bohm experiment}
  \label{fig:ab}
\end{figure}

\paragraph{Setup}
We simulate a quantum particle traveling through a 3D rectangular domain $\Omega = [-5, 5] \times
[-2, 2] \times [-2, 2]$ from which a cylindrical 
solenoid of radius $r_s = 0.5$ is excluded, see Figure \ref{fig:ab} for a 2D representation and Figure \ref{fig:meshes} for a 3D representation with a simplicial mesh.
The particle state is encoded in the quantum state $\upsih$, solution of the time-dependant discrete magnetic Schr\"odinger equation expressed with the discrete covariant gradient: find $\upsih\in\Uh$ such that
\begin{align*}
  (\partial_t \psi_h, \phi_h)_{L^2(\DOM)} + a_h(\upsih, \uphih) & =  \ell(\uphih), \qquad \text{ in } \DOM,\\
  \upsih(t=0) &= \Ih \psi_{0},  \qquad \text{ in } \Gamma.
\end{align*}
We use a Crank-Nicholson time scheme with homogeneous Dirichlet boundary conditions on all boundaries (including outer walls and solenoid surface).
The magnetic vector potential outside the solenoid is the classical irrotational field,
\begin{equation}
  \vA(x, y,z) \eqbydef \frac{\Phi}{2\pi} \frac{1}{x^2 + y^2} \begin{pmatrix} -y \\ x \\ 0 \end{pmatrix}, \qquad (x,y,z)\in \DOM,
\end{equation}
where $\Phi$ is the total magnetic flux confined inside the solenoid. Note that $\ROT \vA = \mathbf{0}$ 
everywhere in $\Omega$ by construction, so \emph{the particle experiences no local magnetic force} 
in the domain; the entire effect is encoded in the topology of $\Omega$ and the circulation of 
$\vA$ around the solenoid.

The initial wavepacket is a Gaussian centered at $x_0 = -3.5$ propagating in the $+x$ direction,
\begin{equation}
  \psi_0(x,y,z) = \exp\!\left( -\frac{(x - x_0)^2 + y^2 + z^2}{2\sigma^2} \right) e^{ik_0 x},
  \qquad \sigma = 0.8,\quad k_0 = 3.
\end{equation}
As the wavepacket encounters the solenoid, it naturally splits into two partial waves propagating along both sides (top and bottom, see Fig. \ref{fig:ab}), which then recombine downstream to form an interference pattern. 

\paragraph{Theoretical Interference Pattern.} The presence of the vector potential $\vA\ne \vzero$ induces a relative phase shift $\Delta\phi$ between the top and bottom partial waves.
This shift is given by the line integral of $\vA$ along the closed loop $\Gamma$ formed by the two paths.
By Stokes' theorem, this is exactly the enclosed magnetic flux:
\[
  \Delta\phi = \oint_{\gamma_{\mathrm{top}}} \vA \cdot d\vl - \oint_{\gamma_{\mathrm{bottom}}} \vA \cdot d\vl = \oint_\Gamma \vA \cdot d\vl = \iint_{\DOM_S} \ROT \vA \cdot d\vS = \iint_{\DOM_S} \vB \cdot d\vS = \Phi.
\]

Let us analyze the expected intensity $I(y)$ on a vertical detection screen downstream, at $x=3.0$. At the exact center of the screen ($y=0$), the geometric paths are perfectly symmetric, see Fig. \ref{fig:ab}.
Let $\psi_s$ denote the local amplitude acquired via each path in the absence of a magnetic field.
When turning on the vector potential, the total wavefunction $\psi_{end}$ at the center is the coherent superposition of the two paths endowed with their respective AB phases: $\psi_{end}(0) = \psi_s e^{i \int_{\gamma_{\text{top}}} \vA \cdot d\vl} + \psi_s e^{i \int_{\gamma_{\text{bottom}}} \vA \cdot d\vl}$.
Factoring out the global phase of the bottom path, which vanishes when taking the squared modulus, the resulting local intensity on the screen is given by $I(0) = |\psi_s e^{i\Phi} + \psi_s|^2$.
We compare two cases: When $\Phi = 0$, i.e., no flux. The phase shift is $\Delta\phi = 0$, leading to $I(0) = 4|\psi_s|^2$, a \emph{constructive} interference, generating a global maximum (a central peak) at $y=0$. When $\Phi = \pi$,  the phase shift is $\Delta\phi = \pi$, leading to $I(0) = 0$, a \emph{destructive} interference. The probability of finding the particle at the exact center is strictly zero, and the displaced probability creates two symmetric peaks bounding this central node.

Furthermore, the local amplitude $\psi_s$ arriving at the screen is expected to be drastically smaller than the initial amplitude ($|\psi_0| = 1$). This amplitude drop is physically consistent and stems from three combined effects: the natural quantum dispersion of the Gaussian wavepacket over time, the strong backward scattering (reflection) of the wave upon hitting the impenetrable Dirichlet cylinder, and the 3D geometric spreading of the wave in the domain.

\paragraph{Results.}

\begin{figure}[htbp]
  \centering
    \includegraphics[width=1.0\textwidth]{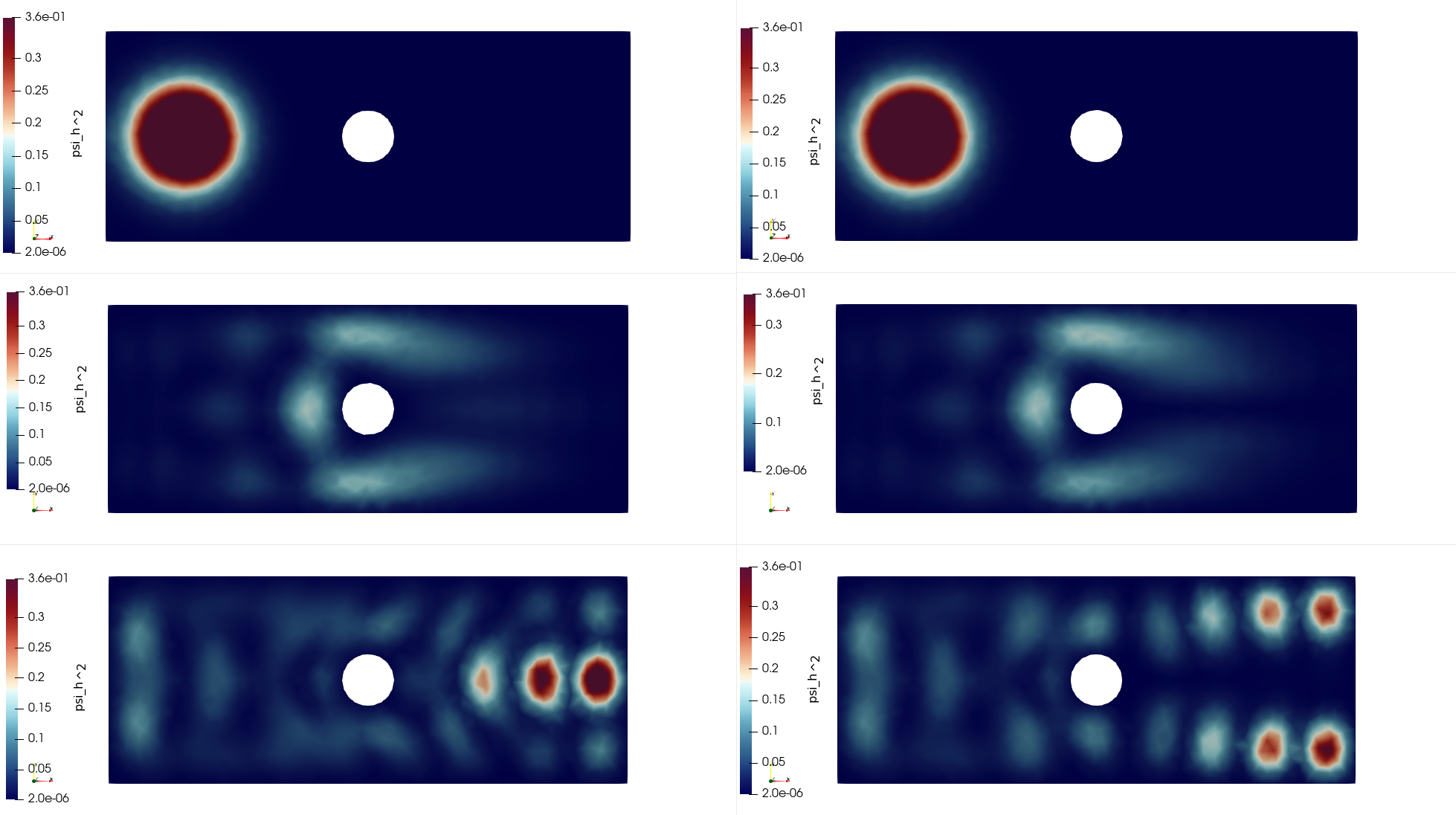} 
  \caption{Aharanov-Bohm test case : (Top) probability density at times $t=0$, (middle) $=0.8$ and $t=1.45$ ; (Left column) No magnetical flux $\Phi=0$, (Right column) with magnetic flux $\Phi=\pi$.}
  \label{fig:ab.chrono}
\end{figure}

\begin{figure}[htbp]
  \centering
    \includegraphics[width=1.0\textwidth]{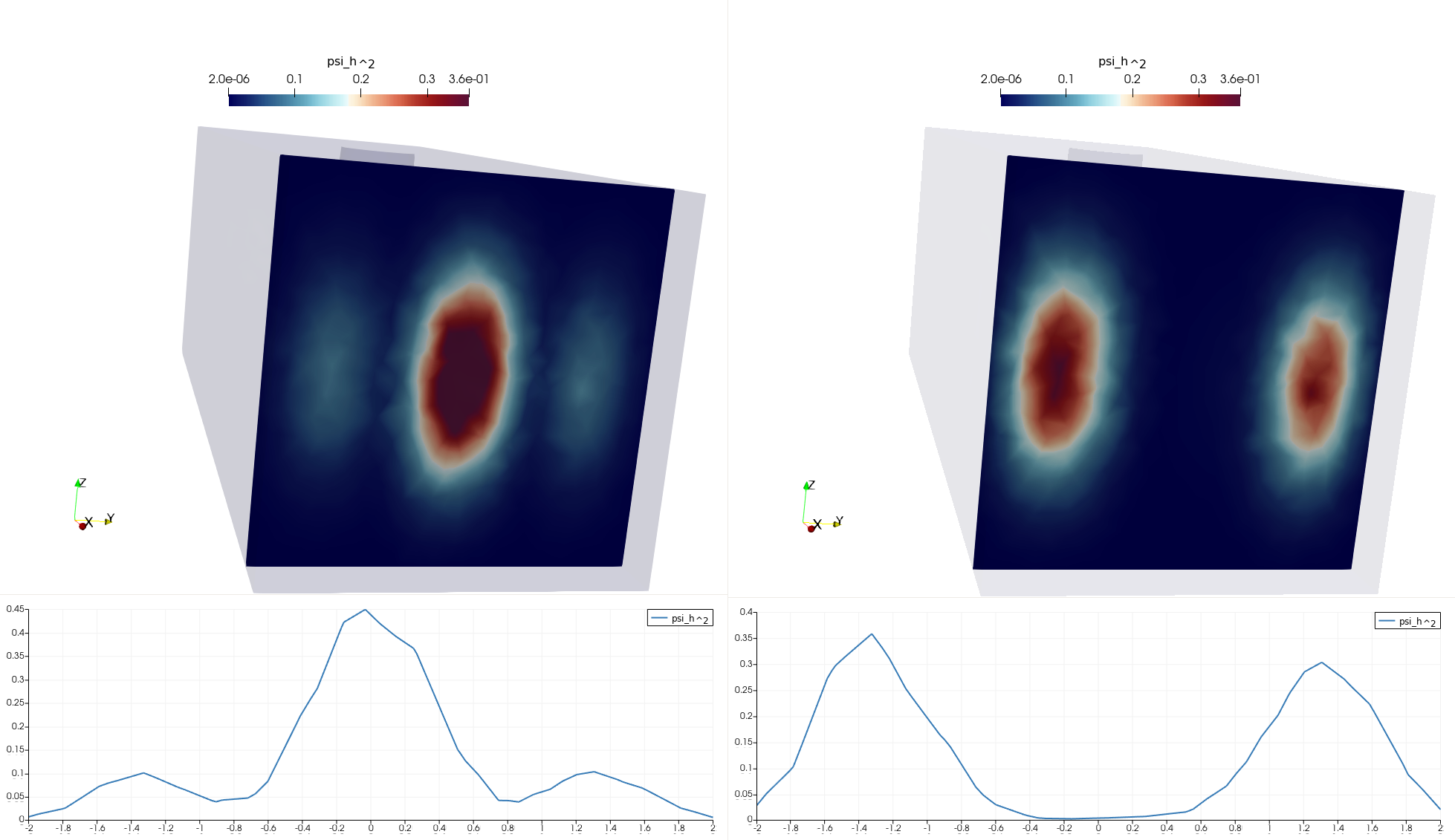} 
  \caption{Aharanov-Bohm test case : Probability density $|\psi_h|^2$ over the screen at time $t=1.49$. (Left) No magnetical flux $\Phi=0$, (Right) with magnetic flux $\Phi=\pi$. The bottom plots shows the value of the density over a line from $(x_{\mathrm{screen}},-2,0)$ to $(x_{\mathrm{screen}},2,0)$.}
  \label{fig:3d.view}
\end{figure}

Figure~\ref{fig:3d.view} displays the computed probability density $|\psi_h|^2(x_{\mathrm{screen}},0,0)$ at time $t = 1.49$.
At the bottom, we display the 1D profiles extracted along the detection line, which match the theoretical predictions.
Indeed, for the zero-flux case ($\Phi = 0$), the profile exhibits an isolated central peak at $y=0$, confirming the constructive interference.
Conversely, for the case $\Phi = \pi$, the central peak completely vanishes and is replaced by a deep central minimum bounded by two symmetric peaks, characterizing the $\pi$ topological phase shift. Additionally, the maximum intensity on the screen correctly reflects the expected amplitude drop, peaking around $\approx 0.45$.

\begin{remark} We also tested the cheaper variant
\[
\GT \uuTF:= -i \GRAD \pT \uuTF - \vAT \, \pT \uuTF,
\]
which reuses the already computed potential reconstruction $\pT$ and avoids solving the local system \eqref{eq:def.GT.adjoint}.
This variant produces similar results for the AB experiment, with similar eigenvalue errors and gauge deviations on all meshes and degrees $k=0,\dots,2$, at the only theoretical cost of a suboptimal $O(h)$ energy-norm bound in \cref{thm:a_priori}.
\end{remark}

\section{Conclusion and Perspectives}

We have proposed a minimal and natural extension of Hybrid High-Order methods to the Schr\"{o}dinger equation with vector potential. 
The scheme relies solely on the existing potential reconstruction operator, allowing for concise derivation of discrete gauge invariance and coercivity.
Future work includes extending to the time-dependent equation, the Pauli model for spin-$\tfrac12$ particles, the Dirac equation, as well as a posteriori error analysis and adaptivity on complex geometries.

\section*{Acknowledgments}
The author acknowledges the use of Gemini 3.1 Pro GenAI to: (i) polish the English translation, (ii) verify proofs of \cref{sec:hho}, and (iii) generate C++ boilerplate  and Python post-processing codes for the numerical experiments.

\appendix

\section{Known HHO results}
\begin{proposition} Under \cite[Assumption 2.4, p. 49]{DiPietro2020}, we have for any $v\in H^{k+2}(T)$,  $s_T(\IT u, \IT u)^{\frac12} \le C h_T^{k+1}|v|_{H^{k+2}(T)}$, where $C$ is independent of $h$, $T$ and $v$.
  \label{prop:sT.bound}
\end{proposition}
\begin{proof} See \cite[Proposition 2.14]{DiPietro2020}. \end{proof}

\begin{theorem} Let a polynomial degree $l \geq 0$, an
  integer $s \in \{0, \ldots, l+1\}$, and a real number $p \in [1, \infty]$ be
  given. Then, for any $X$ element or face of $\Th$, all
  $v \in W^{s,p}(X)$, and all $m \in \{0, \ldots, s\}$,
  \begin{equation}
    |v - \lproj[X]{l} v|_{W^{m,p}(X)}
    \lesssim h_X^{s-m} |v|_{W^{s,p}(X)}.
    \label{eq:proj.approx.cell}
  \end{equation}
  Moreover, if $s \geq 1$, for all $T \in \Th$, all $v \in W^{s,p}(T)$,
  all $F \in \FT$, and all $m \in \{0, \ldots, s-1\}$, it holds that
  \begin{equation}
    h_T^{\frac{1}{p}} |v - \lproj[T]{l} v|_{W^{m,p}(F)}
    \lesssim h_T^{s-m} |v|_{W^{s,p}(T)}.
    \label{eq:proj.approx.trace}
  \end{equation}
  In \eqref{eq:proj.approx.cell} and \eqref{eq:proj.approx.trace}, the hidden constants depend only on
  $d$, $\varrho$, $l$, $s$, and $p$.
\end{theorem}
\begin{proof} See \cite[Theorem 1.45]{DiPietro2020}. \end{proof}

\begin{lemma}[Discrete Poincaré Inequality]
  \label{lem:discrete.Poincare}
  There exists $C_P >0$ depending only on $\DOM, d$ and $\rho$ such that,
  \[
    ||v_h||_{L^2(\DOM)} \le C_P ||\uvh||_{1,h}, \qquad \forall \uvh \in \Uh.
  \]
\end{lemma}
\begin{proof} see \cite[Lemma 2.15]{DiPietro2020} \end{proof}

\bibliographystyle{siamplain}
\bibliography{paper}

\end{document}